\pdfoutput=1
\RequirePackage{ifpdf}
\ifpdf 
\documentclass[pdftex]{sigma}
\else
\documentclass{sigma}
\fi

\usepackage{amstext}
\usepackage{amscd}
\usepackage{soul}

\usepackage{xcolor}
\definecolor{myurlcolor}{rgb}{0,0,0.4}
\definecolor{mycitecolor}{rgb}{0,0.5,0}
\definecolor{myrefcolor}{rgb}{0.5,0,0}
\usepackage{tikz}
\usepackage{tikz-cd}
\usepackage{mathrsfs}
\definecolor{ggray}{gray}{0.25}
 \definecolor{rred}{rgb}{0.5,0,0}
\definecolor{darkgreen}{rgb}{0.0, 0.6, 0.0}

\numberwithin{equation}{section}

\newtheorem{Theorem}{Theorem}[section]
\newtheorem*{Theorem*}{Theorem}
\newtheorem*{TheoremA*}{Theorem~A}
\newtheorem*{TheoremB*}{Theorem~B}
\newtheorem{Corollary}[Theorem]{Corollary}
\newtheorem{Lemma}[Theorem]{Lemma}
\newtheorem{Proposition}[Theorem]{Proposition}
 { \theoremstyle{definition}
\newtheorem{Definition}[Theorem]{Definition}

\newtheorem{Remark}[Theorem]{Remark} }

\newcommand{\be}{\begin{equation}}
\newcommand{\ee}{\end{equation}}
\newcommand{\bea}{\begin{eqnarray}}
\newcommand{\eea}{\end{eqnarray}}

\newcommand{\lra}{\longrightarrow}

\newcommand{\hh}{\mathcal{H}}

\newcommand{\Aut}{\mathcal{A}{\rm ut}}
\newcommand{\Der}{\mathcal{D}{\rm er}}
\newcommand{\ocal}{\mathcal{O}}
\newcommand{\dcal}{\mathcal{D}}
\newcommand{\Jcal}{\mathcal{J}}

\newcommand{\appa}{\mathscr{A}}
\newcommand{\appas}{\mathscr{A}_{\rm sa}}

\newcommand{\stav}{\mathscr{V}}

\newcommand{\m}{\mathfrak{m}}

\newcommand{\G}{\mathrm{G}}

\newcommand{\R}{{\mathbb R}}
\newcommand{\C}{{\mathbb C}}
\newcommand{\N}{{\mathbb N}}

\newcommand{\bH}{{\mathbb H}}
\newcommand{\K}{{\mathbb K}}

\newcommand{\Gl}{{\mathrm{Gl}}}
\newcommand{\Jspin}{{\mathcal{J}\mathrm{Spin}}}

\newcommand{\g}{{\mathfrak{g}}}
\renewcommand{\k}{{\mathfrak k}}

\newcommand{\cf}{{\mathfrak c}}
\newcommand{\h}{{\mathfrak h}}
\renewcommand{\m}{{\mathfrak m}}

\newcommand{\gl}{{\mathfrak{gl}}}

\newcommand{\cR}{\mathcal{R}}
\newcommand{\cA}{\mathcal{A}}

\newcommand{\Gcal}{\mathcal{G}}

\newcommand{\sspan}{{\mathrm{span}}}

\begin{document}
\allowdisplaybreaks

\renewcommand{\thefootnote}{}

\newcommand{\arXivNumber}{2112.09781}

\renewcommand{\PaperNumber}{078}

\FirstPageHeading

\ShortArticleName{Information Geometry, Jordan Algebras, and a Coadjoint Orbit-Like Construction}

\ArticleName{Information Geometry, Jordan Algebras, \\ and a Coadjoint Orbit-Like Construction\footnote{This paper is a~contribution to the Special Issue on Differential Geometry Inspired by Mathematical Physics in honor of Jean-Pierre Bourguignon for his 75th birthday. The~full collection is available at \href{https://www.emis.de/journals/SIGMA/Bourguignon.html}{https://www.emis.de/journals/SIGMA/Bourguignon.html}}}

\Author{Florio M. CIAGLIA~$^{\rm a}$, J\"urgen JOST~$^{\rm bcd}$ and Lorenz J. SCHWACHH\"OFER~$^{\rm e}$}

\AuthorNameForHeading{F.M.~Ciaglia, J.~Jost and L.J.~Schwachh\"ofer}

\Address{$^{\rm a)}$~Department of Mathematics, Universidad Carlos III de Madrid, Legan\'es, Madrid, Spain}
\EmailD{\href{mailto:ciaglia@math.uc3m.es}{ciaglia@math.uc3m.es}}
\URLaddressD{\url{https://researchportal.uc3m.es/display/inv48159}}

\Address{$^{\rm b)}$~Max Planck Institute for Mathematics in the Sciences, Leipzig, Germany}
\EmailD{\href{mailto:jjost@mis.mpg.de}{jjost@mis.mpg.de}}

\Address{$^{\rm c)}$~Center for Scalable Dynamical Systems, Leipzig University, Germany}
\URLaddressD{\url{https://www.mis.mpg.de/jjost/juergen-jost.html}}

\Address{$^{\rm d)}$~Santa Fe Institute for the Sciences of Complexity, New Mexico, USA}

\Address{$^{\rm e)}$~Department of Mathematics, TU Dortmund University, Dortmund, Germany}
\EmailD{\href{mailto:lschwach@math.tu-dortmund.de}{lschwach@math.tu-dortmund.de}}
\URLaddressD{\url{http://riemann.mathematik.tu-dortmund.de/~lschwach/}}

\ArticleDates{Received April 12, 2023, in final form October 09, 2023; Published online October 20, 2023}

\Abstract{Jordan algebras arise naturally in (quantum) information geometry, and we want to understand their role and their structure within that framework. Inspired by Kirillov's discussion of the symplectic structure on coadjoint orbits, we provide a similar construction in the case of real Jordan algebras. Given a real, finite-dimensional, formally real Jordan algebra $\Jcal$, we exploit the generalized distribution determined by the Jordan product on the dual $\Jcal^{\star}$ to induce a pseudo-Riemannian metric tensor on the leaves of the distribution. In particular, these leaves are the orbits of a Lie group, which is the structure group of $\Jcal$, in clear analogy with what happens for coadjoint orbits. However, this time in contrast with the Lie-algebraic case, we prove that not all points in $\Jcal^{*}$ lie on a leaf of the canonical Jordan distribution. When the leaves are contained in the cone of positive linear functionals on $\Jcal$, the pseudo-Riemannian structure becomes Riemannian and, for appropriate choices of $\Jcal$, it coincides with the Fisher--Rao metric on non-normalized probability distributions on a~finite sample space, or with the Bures--Helstrom metric for non-normalized, faithful quantum states of a finite-level quantum system, thus showing a direct link between the mathematics of Jordan algebras and both classical and quantum information geometry.}


\Classification{17C20; 17C27; 17B60; 53B12}

\begin{flushright}
\begin{minipage}{65mm}
 {\em Dedicated to Jean-Pierre Bourguignon}\\
{\em on the occasion of his 75th birthday}
\end{minipage}
\end{flushright}

\noindent
{\bf Personal note of J\"urgen Jost:}\\
I have known Jean-Pierre Bourguignon for many decades, as a dedicated and inspiring differential geometer, as a generous friend, and as an efficient and competent, and above all, always fair, administrator and leader. I am therefore happy that we can dedicate this paper to him on the occasion of his 75th birthday, with great respect for all his achievements.

\renewcommand{\thefootnote}{\arabic{footnote}}
\setcounter{footnote}{0}

\section{Introduction}\label{sec: introduction}

Since their introduction in \cite{J-vN-W-1934} in the context of the foundations of quantum mechanics, Jordan algebras proved to be extremely versatile in both mathematics and physics.
For instance, we mention the link between Jordan algebras, symmetric and harmonic analysis \cite{Chu-2012,Chu-2017,F-K-1994,Koecher-1999,Upmeier-1985}, the connection between Jordan algebras and quantum theories \cite{Baez-2022,F-F-M-P-2014,F-F-I-M-2013}, and the role Jordan algebras play in the reconstruction of quantum theories \cite{Niestegge-2020,W-vdW-2022}.
We also refer to \cite{Iordanescu-2011} for an extensive discussion of different fields of application of Jordan algebras in both mathematics and physics.

Motivated by classical and quantum information geometry, we want to present here another point of view from which to explore Jordan algebras and their mathematics.

Information geometry is a multidisciplinary field of research in which different aspects of mathematics, physics, statistics, and information theory coexist and mutually influence each other.
From a mathematical point of view, information geometry explores the mathematical structures living on suitable manifolds of classical probability distributions or quantum states, known as parametric models, and their relation with information-theoretic and statistical tasks~\cite{Amari-2016,A-N-2000}.
In particular, we have the fundamental concept of an information metric (classical or quantum), a~Riemannian metric tensor on a~manifold $M$ that parameterizes classical probability distributions or quantum states.
The first example of such an information metric is the so-called Fisher--Rao metric tensor, whose introduction in the classical context traces back to the work of Fisher \cite{Fisher-1922}, Mahalanobis \cite{Mahalanobis-1936}, and Rao \cite{Rao-1945}.
The appearance of the Fisher--Rao metric tensor in different applied contexts like population genetics and statistical inference is partially explained by Cencov's theorem \cite{Cencov-1982} which states that, for finite outcome spaces, the Fisher--Rao metric tensor is the only metric tensor which is invariant with respect to the class of Markov kernels which are the most general type of maps respecting the convex geometry of the space of probability distributions.
Cencov's theorem has recently been extended to continuous outcome spaces \cite{A-J-L-S-2015,A-J-L-S-2017,B-B-M-2016,Fujiwara-2023}.

In the quantum case, the situation is not so simple. The first type of quantum information metric appeared in the context of parameter estimation and quantum metrology already at the end of the sixties in the work of Helstrom \cite{Helstrom-1967, Helstrom-1968, Helstrom-1969}.
This Riemannian metric tensor is known as the Bures--Helstrom because it was proved by Uhlmann~\cite{Uhlmann-1976,Uhlmann-1992} that it is the ``infinitesimal'' version of the distance function among states of a von Neumann algebra introduced by Bures~\cite{Bures-1969} to generalize a result by Kakutani \cite{Kakutani-1948}.
Later, it has been proved by Petz \cite{Petz-1996} -- building on previous works by Cencov and Morozowa \cite{C-M-1991} -- that uniqueness is lost in the finite-dimensional quantum case, that is, there is no analogue of Cencov's theorem mentioned above.
This led to an intensive study of the so-called monotone quantum metric tensors generalizing the Fisher--Rao metric tensor to the quantum case \cite{C-DC-L-M-M-V-V-2018,F-K-M-M-S-V-2010,Hasegawa-1995,H-P-1997,L-R-1999,M-M-V-V-2017} and of their information-theoretic and statistical applications \cite{L-Y-L-W-2020,Paris-2009,S-Y-H-2020,T-A-2014,Wootters-1981}.
However, the Bures--Helstrom metric tensor remains the ``best one'' when it comes to quantum metrology because it leads to the tightest version of the quantum Cramer--Rao bound \cite{Fuchs-1996,Paris-2009}.

A common framework to deal with classical and quantum information geometry simultaneously is provided by $W^*$-algebras \cite{C-J-S-2020-02,C-J-S-2020,C-DN-J-S-2023,Jencova-2003,Jencova-2006,Kostecki-2011}.
In this context, it has been recently argued that the Fisher--Rao metric tensor and the Bures--Helstrom metric tensor are related through the Jordan algebra associated with a $W^{*}$-algebra \cite{C-J-S-2020-02,C-J-S-2020, C-DN-J-S-2023}.
These types of Jordan algebras allow for a unification of the classical and quantum states in terms of normal states (i.e., the normalized and normal positive linear functionals) on the algebra.
When the Jordan algebra is the associative Jordan algebra of self-adjoint elements in the commutative $W^{*}$-algebra $\mathcal{L}^{\infty}(\mathcal{X},\mu)$ of complex-valued, $\mu$-essentially bounded (equivalence classes of) functions on the measurable space $(\mathcal{X},\Sigma)$, the normal states are probability measures on $\mathcal{X}$ which are absolutely continuous with respect to $\mu$, and we recover the classical case.
The quantum case is recovered considering the Jordan algebra of self-adjoint elements in the $W^{*}$-algebra $\mathcal{B}(\mathcal{H})$ of bounded linear operators on the complex Hilbert space $\mathcal{H}$, so that the normal states are density operators on $\mathcal{H}$ (i.e., trace-class, positive semidefinite operators with unit trace) which identify quantum states.
In this context, the Jordan product induces a Riemannian metric tensor on suitable parametric models of normal states, which becomes the Fisher--Rao metric tensor when the algebra is $\mathcal{L}^{\infty}(\mathcal{X},\mu)$ and the Bures--Helstrom metric tensor when the algebra is $\mathcal{B}(\mathcal{H})$.

In the finite-dimensional case, it has been observed that the metric tensor on normal states may be obtained as a kind of inverse of a contravariant tensor on the dual space of the Jordan algebra determined by the Jordan product itself \cite{C-J-S-2020}.

Let us recall that setting so that we can subsequently explain how to abstract and generalize it in order to develop new insight into classical and quantum information geometry and at the same time into the structure of Jordan algebras from that of Lie algebras operating on them.
Thus, we let $\appa$ be a finite-dimensional, unital $C^{*}$-algebra, and let $\appas$ be the self-adjoint part of~$\appa$ and $\stav$ be the self-adjoint part of the Banach dual $\appa^{*}$ of $\appa$.
 Of course, there is some duality here, as $\mathbf{a}\in\appas$ can be identified with a real-valued, linear function $l_{\mathbf{a}}$ on $\stav$, via
\[
l_{\mathbf{a}}(\xi) := \xi(\mathbf{a}).
\]
Since $\stav$ is a finite-dimensional Banach space, $\mathbf{a}\mapsto l_{\mathbf{a}}$ yields an isomorphism between $\appas$ and~$\stav^{*}=\appas^{**}$. In other words, the differentials of the linear functions on $\stav$ associated with elements in $\appas$ generate the cotangent space $T_{\xi}^{*}\stav$ at each $\xi$.

The associative product of $\appa$ leads to a commutative product $\{,\}$ and to a non-commutative product $[[\,,\,]]$ on $\appas$, making $\appa_{\rm sa}$ into a \emph{Banach--Lie--Jordan algebra} \cite{A-S-2001}; they are
\begin{gather*}
\{\mathbf{a},\mathbf{b}\} := \frac{1}{2} (\mathbf{ab} + \mathbf{ba}), \qquad
[[\mathbf{a},\mathbf{b}]] := \frac{1}{2\sqrt{-1}} (\mathbf{ab} - \mathbf{ba}).
\end{gather*}
We can then exploit these products to introduce a symmetric and an antisymmetric tensor on the dual of $\appas$.
Both these tensors play a role in the context of the geometry of finite-level quantum systems and their open quantum dynamical evolutions \cite{C-C-I-M-V-2019, C-DC-I-L-M-2017}.

Specifically, we define
\[
\left(\mathcal{R}(\mathrm{d}l_{\mathbf{a}},\mathrm{d}l_{\mathbf{b}})\right)(\xi) := l_{\{\mathbf{a},\mathbf{b}\}}(\xi) = \xi(\{\mathbf{a},\mathbf{b}\}),
\]
and, of course, extend it by linearity to the entire cotangent space, obtaining a symmetric contravariant tensor associated with the Jordan product.

In the same way that the Jordan structure induces the symmetric tensor $\mathcal{R}$, the Lie structure also induces an antisymmetric tensor
\[
\left(\Lambda(\mathrm{d}f_{1},\mathrm{d}f_{2})\right)(\xi) := \xi\left([[\mathrm{d}f_{1}(\xi), \mathrm{d}f_{2}(\xi)]]\right).
\]
Since the Lie algebra of the group $\mathscr{U}$ of unitary elements in $\appa$ may be identified with the space~$\appas$ of self-adjoint elements in $\appa$, the tensor $\Lambda$ may be interpreted as the Kostant--Kirillov--Souriau Poisson tensor associated with the coadjoint action of the unitary group $\mathscr{U}$.
Importantly, for our purposes, also $\mathcal{R}$ can be studied with the help of the action of this Lie algebra.

Returning to the covariant symmetric tensor $\mathcal{R}$, it happens that, on certain orbits of the Banach--Lie group of invertible elements of $\appa$ (of which $\mathscr{U}$ is a subgroup), it makes sense to consider the covariant tensor $\G := \mathcal{R}^{-1}$, which becomes a~Riemannian metric tensor when the above-mentioned orbits lie in the space of positive linear functionals on $\appa$. The result is a~Riemannian metric tensor, which is precisely that used in classical and quantum information geometry.

The observation in \cite{C-J-S-2020} that we have just explained only hinted at the possibility of obtaining Riemannian metric tensors with a coadjoint orbit-like procedure, and the main purpose of this work is to systematically explore this possibility in the more general context of finite-dimensional, formally real Jordan algebras.

In fact, the action of the Lie algebra of $\mathscr{U}$ is reminiscent of what happens in the case of Lie groups and Lie algebras, where the Lie algebra induces a contravariant tensor on its dual that can be inverted on the so-called coadjoint orbits in order to equip them with a symplectic structure known as the Kirillov--Kostant--Souriau (KKS) symplectic form \cite{Kirillov-1962,Kirillov-1976,Kostant-1970,Souriau-1970}.


Therefore, in this paper, we develop a more abstract perspective, moving from the Jordan algebras arising from $C^{*}$-algebras to general Jordan algebras. In particular, this will clarify how much of the above construction actually depends on the ambient $C^{*}$-algebras.
In fact, we shall find that, at least when we deal with finite-dimensional formally real Jordan algebras, the ambient $C^{*}$-algebras play no role.
On one hand, this will provide a deeper explanation of the Fisher--Rao and Bures--Helstrom tensors and their generalizations. On the other hand, from
 this more general perspective, we obtain Riemannian metric tensors on suitable models of positive linear functionals on Jordan algebras that do not necessarily arise in the context of classical and quantum information geometry as the spin factors that are recently being employed in a new formulation of color perception theory \cite{Berthier-2020,B-P-P-2022,B-P-2022,Provenzi-2020,Resnikoff-1974}.

As anticipated above, our guiding question is to what extent we can develop a coadjoint orbit-like construction for Jordan algebras that is analogous to that of Lie algebras, at least in finite dimensions.
Specifically, we start from Kirillov's fundamental observation that a coadjoint orbit $\ocal \subset \g^\star$ of the Lie group $\G$ carries a natural homogeneous symplectic structure \cite{Kirillov-1962,Kirillov-1976}.
Here, $\g$ and $\g^\star$ denote the Lie algebra of $\G$ and its dual.
Kirillov's observation relates algebraic structures to differential geometry and mathematical physics in a deep and very productive way, led to spectacular results in representation theory, classical and quantum mechanics, and it is closely related to geometric quantization \cite{Kirillov-2001}.

As far as we know, no analogue of the coadjoint orbit construction for Lie algebras has been investigated in the case of Jordan algebras, and the main purpose of this work is precisely to fill this gap.
To understand the picture from a more abstract perspective, we need to go somewhat beyond the setting sketched above. We shall start by considering a general, finite-dimensional algebra $\cA$, i.e., a finite-dimensional (real or complex) vector space with a bilinear product $\bullet\colon \cA \times \cA \to \cA$.
At this moment, in contrast to the setting above, no further conditions like associativity or identities of Jacobi/Jordan type are assumed.
We denote with $\cA^\star$ the dual space of $\cA$, and the corresponding pairing is denoted by $\langle \cdot,\cdot \rangle$.
Due to the finite-dimensionality, we have the identification $\cA^{\star\star} \cong \cA$, and
for the tangent and cotangent spaces of $\cA^\star$, we then have $T_\xi
\cA^\star \cong \cA^\star$ and $T^\ast_\xi \cA^\star \cong \cA^{\star\star}
\cong \cA$. We may then represent the product $\bullet$
via
\begin{equation}
 \label{01}
 \langle \xi, a \bullet b\rangle \qquad \text{for }\xi
 \in \cA^\star, \ a, b \in \cA.
\end{equation}
Furthermore, this induces a multiplication on $C^\infty(\cA^\star)$ via
\be \label{02}
f \bullet g(\xi) :=
\langle \xi, d_\xi f
\bullet d_\xi g \rangle,
\ee
and for $f\in C^\infty(\cA^\star)$, we may define its {\em $\cA$-dual vector field of $f$} as
\[
\big(\nabla^{\cA}(f)\big)_\xi(g) :=\{f, g\}_\cA(\xi).
\]
Note that, in general, this is not a gradient because $\bullet$ need not be symmetric, but we use the symbol $\nabla$ because it satisfies many of the formal identities of gradients.

The multiplication in equation \eqref{02} is of course compatible with the product $\bullet$ in the sense that, identifying $a\in \cA$ with the linear functional $ f_a(\xi) := \langle \xi, a\rangle$, the inclusion $\cA \hookrightarrow C^\infty(\cA^\star)$ is an algebra monomomorphism
\[
f_a \bullet f_b = f_{a \bullet b},
\]
justifying the symbol $\bullet$ to denote the multiplication on both $\cA$ and $C^\infty(\cA^\star) \supset \cA$. The automorphisms of $\cA$, i.e., the linear isomorphisms $g\colon \cA \to \cA$ with $g(a \bullet b) = (ga) \bullet (gb)$, form a~Lie group.
The Lie
algebra of that group consists of the derivations, i.e., the linear maps~$d \in \gl(\cA)$ with $d(a \bullet b) = (da) \bullet b + a \bullet
(db)$.
Moreover, we have the structure Lie group $\G(\cA)$ whose Lie algebra is generated by left multiplications, i.e., by the maps $l_a\colon (b \mapsto a \bullet b) \in \gl(\cA)$.

Therefore, even though we do not assume $\cA$ to be a Lie algebra, there is a Lie algebra that is naturally associated to $\cA$, and we may hope to use its theory to gain insight about $\cA$ itself.
This works to some extent, but a problem arises from the fact that the {\em $\cA$-dual distribution} $\hh^\cA$ on $\cA^\star$ defined by
\[
\hh^\cA_\xi := \big\{ \big(\nabla^\cA(f)\big)_\xi \mid f \in C^\infty(\cA^\star)\big\} 
\subset T_\xi \cA^\star.
\]
is in general not integrable.
We recall that, in the Lie algebra case, $\hh^\cA$ is integrable, and its leaves are precisely the coadjoint orbits which carry a symplectic structure induced by equation~\eqref{02}.
Because $\hh^\cA$ is not integrable for general algebras, we cannot expect the same level of generality as in Lie algebras.

However, the results become stronger if we also assume some additional structures on $\cA$.
Associativity already gives us some leverage, but the more specific case that we are interested in here is when $\cA$ is a Jordan algebra.
Our strategy then is to combine this Jordan structure, and the identities resulting from it, with the Lie algebra structure that we just have identified.

Thus, we consider a finite-dimensional real Jordan algebra $\cA = \Jcal$.
In this case, we can also define an extended structure Lie algebra $\hat \g(\Jcal)$ that maps surjectively onto $\g(\Jcal)$ \cite{Bertram-2000,Koecher-1999,L-L-2022}.
This algebra is the direct sum $\Der_0(\Jcal) \oplus \Jcal$ of the inner derivations (those generated by left multiplications $l_x$ with algebra elements $x$) and $\Jcal$ itself.
The Lie bracket on $\Der_0(\Jcal)$ is the commutator, while, for $x,y \in \Jcal$, it is $[x,y] := [l_x, l_y] \in \Der_0(\Jcal)$, and, for $d\in \Der_0(\Jcal)$ and $x \in \Jcal$, it is $[d, x] = - [x, d] := d(x) \in
\Jcal$.
Therefore, putting $\k =\Der_0(\Jcal)$, $\m_{\Jcal}= \Jcal$, we have $[\k, \k] \subset \k$, $[\k, \m_{\Jcal}] \subset \m_{\Jcal}$, $[ \m_{\Jcal}, \m_{\Jcal}] = \k$, i.e., we have a transvective symmetric pair (see Definition~\ref{df:symm}) which provides us with further structure to work with.

The generalized distribution $\hh^\Jcal$ on $\Jcal^\star$ is still not integrable in general, but the bilinear form~$\Gcal_\xi$ induced by equation \eqref{01} on $\hh^\Jcal_\xi$ is symmetric, and therefore it defines a pseudo-Riemannian metric on the $ \m_{\Jcal}$-regular part of each $\G(\Jcal)$-orbit $\ocal^{\rm reg}_{\m_{\Jcal}} \subset \Jcal^\star$ (see Definition~\ref{def:regular}).

On $\Jcal$ there is the symmetric, bilinear form $\tau(x,y) := \operatorname{tr} l_{\{x, y\}}$, which is also associative with respect to the Jordan product, and we can use it to decompose $\Jcal$ \cite[p.~59]{Koecher-1999}.
Specifically, in the case of {\em positive Jordan algebras}, $\tau$ is positive definite so that there is a canonical identification~${\Jcal \cong \Jcal^\star}$, and we can also derive additional properties from the above-mentioned decomposition.
Each~$\xi \in \Jcal^\star$ has a {\em spectral decomposition} associating with $\xi$ its {\em spectral coefficients $(\lambda_i)_{i=1}^r \in \R^r$}, where~$r$ denotes the rank of $\Jcal$. The pair $(n_+, n_-)$ counting the number of positive and negative spectral coefficients is called the {\em spectral signature of $\xi$}. We then show the following:

\begin{TheoremA*}[cf.\ Theorem \ref{thm:orbits}]
If $\Jcal$ is a positive simple real Jordan algebra, then the orbits of the structure group $\G(\Jcal)$ consist of all elements with the same spectral signature.
\end{TheoremA*}

We also characterize the {\em regular }points (i.e., those where the
generalized distribution $\hh^\Jcal$ is integrable) in such a $\G(\Jcal)$-orbit in Theorem \ref{thm:main} and describe the pseudo-Riemannian metric $\Gcal_\xi$ at each regular point $\xi \in \Jcal^\star$ in Proposition \ref{prop:metric-Jordan}.
Specifically, let $\Omega_\Jcal$ denote the {\em cone of squares of $\Jcal$}, i.e., the interior of the set $\big\{x^2 \mid x \in \Jcal\big\}$. Then $\Omega_\Jcal$ is the $\G(\Jcal)$-orbit of the identity ${\bf 1}_\Jcal$.
 The characterizations in Theorem \ref{thm:main}, Proposition \ref{prop:inner-product}, and Proposition \ref{prop:metric-Jordan} lead to the following remarkable description:

\begin{TheoremB*} Let $\Jcal$ be a positive simple real Jordan
 algebra. Then all points of a $\G(\Jcal)$-orbit $\ocal \subset \Jcal^\star$ are $\m_{\Jcal}$-regular iff $\ocal \subset \overline{\Omega}_\Jcal$ or $\ocal \subset -\overline{\Omega}_\Jcal$. The form $\Gcal$ on $\ocal$ is positive definite in the first and negative definite in the second case, thus defining a Riemannian metric on $\ocal$ which is invariant with respect to the action of the automorphism group of $\Jcal$.
\end{TheoremB*}

For all regular $\xi \notin \pm \overline{\Omega}_\Jcal$ the form $\Gcal_\xi$ is indefinite, so the definiteness of $\Gcal_\xi$ gives a new characterization of $\Omega_\Jcal$.
We provide descriptions of the orbits $\ocal$ and the metric $\Gcal$ for the standard examples of positive simple real algebras. Moreover, the above results easily generalize to the case of non-simple positive Jordan algebras, as these algebras are direct sums of positive simple Jordan algebras.

This work is structured as follows.
In Section \ref{sec:prelim}, we set the notation and recall some standard results on generalized distributions and group actions on manifolds that are needed.
In Section \ref{sec:structure-general}, we discuss the structure group and the generalized distributions on an arbitrary finite-dimensional algebra $\cA$.
In Section \ref{sec: Jordan algebras}, we focus on Jordan algebras, prove the main results, and discuss relevant examples.
Let us stress that some results recalled in Section \ref{sec: Jordan algebras} are well known to researchers working with Jordan algebras, but we decided to recall them nonetheless in order to make the work as self-contained
as possible for readers, perhaps coming from information geometry, who are not familiar with Jordan algebras.
Finally, in Section \ref{sec: conclusions}, we discuss our results in relation with some important established results on the mathematics and geometry of Jordan algebras.

\section{Preliminary material} \label{sec:prelim}

\subsection{Notational conventions} \label{sec:notation}

For a finite-dimensional (real or complex) vector space $V$, we denote by $\Gl(V)$ the Lie group of linear automorphisms, whose Lie algebra $\gl(V)$ consists of all linear endomorphisms of $V$.

For finite-dimensional vector spaces $V$, $W$ and $U$, we define the contraction map
\be \label{eq:bracket1}
\langle \cdot, \cdot \rangle_V\colon\ V^\star \otimes V \otimes U \longrightarrow U, \qquad \langle \alpha, v \otimes u \rangle_V := \alpha(v) u,
\ee
where $V^\star$ denotes the dual space of $V$, and we shall usually omit the subscript if this causes no ambiguity. The notation \eqref{eq:bracket1} can also be used to denote maps
\[
V^\star \otimes \Lambda^k V^\star \otimes U \longrightarrow \Lambda^{k-1} V^\star \otimes U, \qquad V^\star \otimes S^k V^\star \otimes U \longrightarrow S^{k-1} V^\star \otimes U,
\]
denoting by $S^k V^\star$ the symmetric $k$-forms, since $\Lambda^k V^\star \hookrightarrow V \otimes \Lambda^{k-1} V^\star$ and $S^k V^\star \hookrightarrow V \otimes S^{k-1} V^\star$ are canonically included.

Finally, the dual of a linear map $\phi\colon V \to W$ is denoted as
\[
\phi^\star\colon\ W^\star \longrightarrow V^\star, \qquad \langle \theta, \phi(v) \rangle_W = \langle \phi^\star \theta, v \rangle_V.
\]


\subsection{Generalized distributions}

Let $M$ be a finite-dimensional, real smooth manifold.
A {\em generalized distribution on $M$} is a family $\dcal = (\dcal_p)_{p \in M}$ of subspaces $\dcal_p \subset T_pM$.
We let $\Gamma(\dcal)$ be the set of (smooth) vector fields $X$ on $M$ with $X_p \in \dcal_p$ for all $p$, and we call $\dcal$ {\em smooth} if for each $v \in \dcal_p$ there is a vector field $X \in \Gamma(\dcal)$ with $X_p = v$.

Given a smooth generalized distribution $\dcal$, we define the {\em Frobenius tensor ${\mathcal F}_p$ at $p \in M$} as
\[
{\mathcal F}_p\colon\ \Lambda^2 \dcal_p \longrightarrow T_pM/\dcal_p, \qquad (X_p, Y_p) \longmapsto [X,Y]_p \mod \dcal_p
\]
for $X, Y \in \Gamma(\dcal)$. It is straightforward to verify that ${\mathcal F}_p(X, Y)$ depends on $X_p$ and $Y_p$ only, i.e., ${\mathcal F} = ({\mathcal F}_p)_{p \in M}$ is a well-defined tensor field. We also define the generalized distribution
\[
{}[\dcal, \dcal]_p := \dcal_p + \{ [X, Y]_p \mid X, Y \in \Gamma(\dcal)\},
\]
so that the image of ${\mathcal F}_p$ is $[\dcal, \dcal]_p/\dcal_p$.

\begin{Definition}
We call a smooth generalized distribution $\dcal$ {\em involutive at $p \in M$}, if ${\mathcal F}_p = 0$ or, equivalently, if $[\dcal, \dcal]_p = \dcal_p$, and we call it {\em involutive}, if this holds for every $p$.

Furthermore, an (immersed) submanifold $N \subset M$ with $T_pN = \dcal_p$ for all $p \in N$ is called an {\em integral leaf of $\dcal$}. If there is an integral leaf of $\dcal$ containing $p$, then we call $\dcal$ {\em integrable at $p \in M$} and call $p$ an {\em integral point of $\dcal$}; if this is the case for each $p \in M$, then we call $\dcal$ {\em integrable}.
\end{Definition}

Clearly, if $\dcal$ is integrable (at $p$), then it is also involutive (at $p$); according to Frobenius' theorem, the converse of this statement holds if $\dcal$ has constant rank. However, if the rank of $\dcal$ is non-constant, then the converse may fail to hold \cite{Sussmann-1973}.

\subsection[G-manifolds]{$\boldsymbol{G}$-manifolds}\label{subsec: G-manifolds}

Let $G$ be a finite-dimensional, real Lie group with identity element $e$, Lie algebra $\g\cong T_{e}G$, and let $M$ be a $G$-manifold, i.e., a finite-dimensional, real smooth manifold with a smooth left action $\pi\colon G \times M \to M$, $(g, p) \mapsto g \cdot p$. For $p \in M$ we define the {\em stabilizer of $p$} to be the subgroup
\[
H_p := \{ g \in G \mid g \cdot p = p\} \subset G \qquad \text{with Lie algebra} \quad \h_p \subset \g.
\]
Evidently, $H_{g \cdot p} = gHg^{-1}$, and $\h_{g \cdot p} = \operatorname{Ad}_g(\h_p)$, so that the stabilizer on each $G$-orbit is unique up to conjugation.

We define the {\em orbit distribution on $M$} by
\[
\dcal^\g_p := \{(X_\ast)_p \mid X \in \g\},
\]
where $X_\ast$ denotes the action field on $M$. Evidently, $\dcal^\g$ is integrable, as the $G$-orbits are integral leaves of $\dcal^\g$.
Moreover,
\[ 
X \in \h_p \quad \Leftrightarrow \quad (X_\ast)_p = 0.
\]
Also recall that the map $X \mapsto X_\ast$ is a anti-homomorphism of Lie algebras, i.e.,
\be \label{eq:action-comm}
{}[X, Y]_\ast = - [X_\ast, Y_\ast].
\ee

For any linear subspace $\m \subset \g$, we define the smooth generalized distribution $\dcal^\m$ by
\be \label{eq:Dm}
\dcal^\m_p := \{(X_\ast)_p \mid X \in \m\} \subset \dcal^\g_p = T_p(G \cdot p),
\ee
and evidently,
\be \label{eq:Dm-criterion}
(X_\ast)_p \in \dcal^\m_p \quad \Leftrightarrow \quad X \in \m + \h_p.
\ee

\begin{Lemma} \label{lem:involutive}
Let $\m \subset \g$ be a linear subspace. Then the following are equivalent:
\begin{enumerate}\itemsep=0pt
\item[$(1)$] $\dcal^\m$ is involutive at $p \in \ocal$,
\item[$(2)$] $[\m, \m] \subset \m + \h_p$,
\item[$(3)$] $\dcal^{[\m, \m]}_p \subset \dcal^\m_p$.
\end{enumerate}
\end{Lemma}

\begin{proof}
For $X, Y \in \m$, ${\mathcal F}_p(X_p, Y_p) = 0$ iff $[X_\ast, Y_\ast]_p \in \dcal^\m_p$, which, by equation \eqref{eq:action-comm}, is the case iff $([X, Y]_\ast)_p \in \dcal^\m_p$, and, by equation \eqref{eq:Dm-criterion}, this is the case iff $[X,Y] \in \m + \h_p$, showing the equivalence of the first two conditions.

The second condition is equivalent to saying that for each $X \in [\m,\m]$ there is a $Y \in \m$ such that~$X - Y \in \h_p$ or, equivalently, that for each $X \in [\m,\m]$ there is a $Y \in \m$ such that~$(X_\ast)_p = (Y_\ast)_p$, and this is evidently equivalent to the third condition.
\end{proof}

\begin{Definition} \label{def:regular}
Let $M$ be a $G$-manifold and $\m \subset \g$ a linear subspace. We call $p \in M$ an {\em $\m$-regular point} if $\dcal^\m_p = T_p\ocal$, where $\ocal \subset M$ is the $G$-orbit of $p$. The subset of $\m$-regular points in $\ocal$ is denoted by
$\ocal^{\rm reg}_{\m} \subset \ocal$.
\end{Definition}

As the rank of $\dcal^\m_p$ is a lower semicontinuous function in $p$, $\ocal^{\rm reg} \subset \ocal$ is open (but possibly empty).
As we shall see in later sections, $\ocal^{\rm reg}_{\m}$ may be a proper subset of $\ocal$ and is not necessarily connected.

\begin{Corollary} \label{cor:integrable}
Suppose that $\m \subset \g$ is a linear subspace such that
\be \label{eq:condition-integrable}
\g = \m + [\m,\m].
\ee
Then for each $p \in M$ the following are equivalent:
\begin{enumerate}\itemsep=0pt
\item[$(1)$] $p$ is involutive,
\item[$(2)$] $p$ is integrable,
\item[$(3)$] $p$ is an $\m$-regular point.
\end{enumerate}
In this case, the maximal integral leaf through $p$ is the connected component of $p$ in
\[
\ocal^{\rm reg}_{\m}  \subset \ocal  =  \G \cdot p.
\]
\end{Corollary}

\begin{proof}
By Lemma \ref{lem:involutive}, $p$ is integrable iff $\dcal^{[\m, \m]}_p \subset \dcal^\m_p$, and as $\g = \m + [\m,\m]$, this is the case iff~$\dcal^\m_p = \dcal^\g_p$, i.e., iff $p$ is $\m$-regular. It follows that any integral leaf through $p$ must be (an open subset of) $\ocal^{\rm reg}_{\m} \subset \ocal$, whence the maximal (connected) leaf through $p$ is its path component in~$\ocal^{\rm reg}_{\m}$.
\end{proof}

\begin{Definition} \label{df:symm}
A {\em symmetric pair} is a pair $(\g, \k)$ of Lie algebras with a decomposition $\g = \k \oplus \m$ satisfying \be \label{eq:bracket}
{}[\k, \k] \subset \k, \qquad [\k, \m] \subset \m, \qquad [\m, \m] \subset \k.
\ee
We call this pair {\em transvective}, if $\g$ is generated by $\m$ as a Lie algebra, i.e., if $[\m, \m] = \k$.
\end{Definition}

Clearly, equation \eqref{eq:bracket} is equivalent to saying that the involution $\sigma\colon \g \to \g$ with $\k$ and $\m$ as the $(+1)$- and $(-1)$-eigenspace, respectively, is a Lie algebra automorphism.

\section[Structure groups and canonical distributions on duals of algebras]{Structure groups and canonical distributions\\ on duals of algebras}\label{sec:structure-general}

In this section, we shall consider a finite-dimensional algebra $\cA$, by which we simply mean a~finite-dimensional (real or complex) vector space with a~bilinear product $\bullet\colon \cA \times \cA \to \cA$, i.e., a~constant $(2,1)$-tensor.
We do not assume any further conditions on this multiplication such as associativity, Jacobi or Jordan identities, but we shall later discuss the general definitions in each of these cases.

The dual of $\bullet$ is a map $\cA^\star \to \cA^\star \otimes \cA^\star$, and as $T_\xi \cA^\star \cong \cA^\star$, we may regard this as a linear bivector field on $\cA^\star$:
\be \label{eq:R-bivectorfield}
\cR^\cA \in \Gamma(\cA^\star, T\cA^\star \otimes T\cA^\star), \qquad \cR^\cA_\xi (a, b) := \langle \xi, a \bullet b\rangle
\ee
for all $a,b \in \cA$.
Therefore, there is an induced multiplication on the space $C^\infty(\cA^\star)$ of real-valued, smooth functions on $\cA^{\star}$ which by abuse of notation we also denote by $\bullet$, given by
\be \label{eq:A-Poisson}
f \bullet g(\xi) := \big(\cR^\cA\big)_\xi (d_\xi f, d_\xi g) = \langle \xi, d_\xi f \bullet d_\xi g \rangle
\ee
with the canonical identification $d_\xi f, d_\xi g \in T^\ast_\xi \cA^\star \cong \cA^{\star\star} \cong \cA$. In particular, regarding $ \cA \subset C^\infty(\cA^\star) $ as the set of linear functions $f_a(\xi) := \langle \xi, a\rangle$, equation \eqref{eq:A-Poisson} implies
\[ 
f_a \bullet f_b = f_{a \bullet b},
\]
justifying the ambiguous use of the symbol $\bullet$. 
Contraction in the first entry yields a linear map
\be \label{eq:sharp-xi}
\sharp_\xi\colon\ T^\star_\xi \cA^\star \lra T_\xi \cA^\star, \qquad \theta \longmapsto \theta \lrcorner \cR^\cA_\xi,
\ee
whose image we call the {\em $\cA$-dual distribution} $\hh^\cA$ on $\cA^\star$ by
\be \label{eq:grad-dist}
\hh^\cA_\xi := \sharp_\xi(T_\xi^\star \cA^\star) \subset T_\xi \cA^\star.
\ee

\begin{Definition} \label{def:structure-Lie} For a finite-dimensional algebra $(\cA, \bullet)$ we define the following:
\begin{enumerate}\itemsep=0pt
\item[(1)]
For $a \in \cA$ we let $l_a \in \gl(\cA)$ be the map $(b \mapsto a \bullet b) \in \gl(\cA)$.
\item[(2)]
The {\em structure Lie algebra of $\cA$ }is the Lie subalgebra $\g(\cA) \subset \gl(\cA)$ generated by $\m_\cA := \{l_a \mid a \in \cA\}$.
\item[(3)]
The {\em structure Lie group of $\cA$ }is the connected Lie subgroup $\G(\cA) \subset \Gl(\cA)$ with Lie algebra~$\g(\cA)$.
\item[(4)]
A {\em derivation of $\cA$ }is a linear map $d \in \gl(\cA)$ with $d(a \bullet b) = (da) \bullet b + a \bullet (db)$.
\item[(5)]
An {\em automorphism of $\cA$ }is a linear isomorphism $g\colon \cA \to \cA$ with $g(a \bullet b) = (ga) \bullet (gb)$.
\end{enumerate}
\end{Definition}

It is straightforward to verify that the automorphisms and derivations form a regular Lie subgroup and a Lie subalgebra $\Aut(\cA) \subset \Gl(\cA)$ and $\Der(\cA) \subset \gl(\cA)$, respectively, called the {\em automorphism group }and {\em derivation algebra of $\cA$}, respectively. In fact, $\Der(\cA)$ is the Lie algebra of $\Aut(\cA)$. Moreover, $g \in \Aut(\cA)$ and $d \in \Der(\cA)$ iff for all $a \in \cA$ we have
\be \label{eq:der-leftaction}
g l_a g^{-1} = l_{ga}, \qquad [d, l_a] = l_{da}.
\ee
That is, the adjoint action of $\Aut(\cA)$ and $\Der(\cA)$ on $\gl(\cA)$ preserves the subspace $\m_\cA$ and hence the structure Lie algebra $\g(\cA)$ and structure Lie group $\G(\cA)$.\footnote{Actually, it would be more accurate to call $\g(\cA)$ and $\G(\cA)$ the {\em left-}structure Lie algebra and group, respectively, and to define the {\em right-}structure Lie algebra and group analogously. However, for simplicity we shall restrict ourselves to the left-structure case, as the right-structure case can be treated in complete analogy.}

For $\xi \in \cA^\star$ and $a,b \in \cA$, we have
\[
\langle l_a^\ast(\xi), b \rangle = \langle \xi, l_a(b)\rangle = \langle \xi, a \bullet b\rangle \stackrel{(\ref{eq:R-bivectorfield})}= \big\langle a \lrcorner \cR^\cA_\xi, b\big\rangle \stackrel{(\ref{eq:sharp-xi})}= \langle \sharp a, b
\rangle_\xi,
\]
so that $l_a^\ast = \sharp a \in \Gamma(\cA^\star, T^\ast \cA^\star)$, regarded as a linear $1$-form on $\cA^\star$. It follows that
\be \label{eq:grad=dist}
\hh^\cA_\xi = \dcal^{\m_\cA}_\xi
\ee
with $\hh^\cA_\xi$ from equation \eqref{eq:grad-dist} and $\m_\cA$ from Definition \ref{def:structure-Lie}, regarding $\cA^\star$ as a $\G(\cA)$-manifold via the dual representation $\imath\colon \G(\cA) \to \Gl(\cA^\star)$.

More generally, for a smooth function $f \in C^\infty(\cA^\star)$, we define the {\em $\cA$-dual vector field of $f$} as
\be \label{eq:grad-A}
\nabla^{\cA}(f) := \sharp df \in \mathfrak{X}(\cA^\star).
\ee

Since $\sharp a = \sharp f_a = \nabla^\cA(f_a)$, it follows that the $\cA$-dual distribution may also be characterized by
\[
\hh^\cA_\xi = \big\{ \big(\nabla^\cA(f)\big)_\xi \mid f \in C^\infty(\cA^\star)\big\},
\]
and unwinding the definitions, it easily follows that for $f, g \in C^\infty(\cA^\star)$ that
\[
\big(\nabla^\cA(f)\big)(g) = \{f, g\}_\cA.
\]

\begin{Remark} If the multiplication $\bullet$ is symmetric (e.g., if $\cA$ is a Jordan algebra), the dual vector field $\nabla^\cA(f)$ is usually referred to as the {\em gradient vector field of $f$}, while in the case of an antisymmetric multiplication $\bullet$ (e.g., if $\cA$ is a Lie algebra), it is called the {\em Hamiltonian vector field of $f$}. That is, the term {\em $\cA$-dual vector field }subsumes both cases.

We wish to caution the reader that in case of a skew-symmetric multiplication $\bullet$, the notation~$\nabla^{\cA}(f)$ for the Hamiltonian vector does not match the standard convention. We use it nevertheless to unify our notation.
\end{Remark}

%

If $g \in \Aut(\cA)$ is an automorphism, then by equation \eqref{eq:der-leftaction} the action of $g^\ast$ on $\cA^\star$ preserves the distribution $\hh^\cA$ and hence permutes integral leaves of equal dimensions, preserving $\m_\cA$-regular points.

As it turns out, if the product $\bullet$ is symmetric or skew-symmetric, then there is a canonical bilinear pairing on $\hh^\cA$.

\begin{Proposition} \label{prop:inner-product}
Let $(\cA, \bullet)$ be a finite-dimensional real algebra such that $\bullet$ is symmetric $($skew-symmetric, respectively$)$.
Then on $\hh^\cA_\xi = \dcal^{\m_\cA}_\xi \subset T_\xi \cA^\star$ there is a canonical non-degenerate symmetric $($skew-symmetric, respectively$)$ bilinear form, given by
\be \label{eq:inner-productA}
\Gcal_\xi (l_a^\ast(\xi), l_b^\ast(\xi)) := \langle \xi, a \bullet b \rangle.
\ee
Furthermore, $\Gcal$ is preserved by the action of the automorphism group $\Aut(\cA)$.
\end{Proposition}

\begin{proof} Since $\langle \xi, a \bullet b \rangle = \langle l_a^\ast(\xi), b\rangle = \pm \langle l_b^\ast(\xi), a\rangle$, where the sign $\pm$ depends on the symmetry or skew-symmetry of $\bullet$, it follows that $\Gcal$ is indeed well defined and non-degenerate.

Finally, if $g \in \Aut(\cA)$ is an automorphism, then equation \eqref{eq:der-leftaction} implies that
\[
g^\ast l_a^\ast \big(g^{-1}\big)^\ast = l_{g^{-1}a}^\ast
\]
and from here, the invariance of $\Gcal$ under the action of the automorphism group follows.
\end{proof}

\begin{Definition} \label{def:regular-A} Let $\cA$ be an algebra. A $\G(\cA)$-orbit $\ocal \subset \cA^\star$ is called {\em $\m_\cA$-regular } if $\ocal^{\rm reg}_{\m_\cA} = \ocal$.
\end{Definition}

We shall now give classes of examples of these notions.

{\bf 1.\ Lie algebras.}
Let $(\cA, \bullet) = (\g, [\cdot, \cdot])$ be a Lie algebra. Then the induced section $\Lambda := \cR^\g \in \Gamma\big(\g^\star, \Lambda^2 T\g^\star\big)$ from equation \eqref{eq:R-bivectorfield} is a skew-symmetric bi-vector field, and the Jacobi identity implies that the Schouten--Nijenhuis bracket $[\Lambda, \Lambda] \in \Gamma(\g^\star, \Lambda^3 T\g^\star)$ vanishes \cite{Lichnerowicz-1977,Tulczyjew-1974}, so that~$\Lambda$ defines a linear Poisson structure $\{\cdot,\cdot\}$ on $\g^\star$, also known as the {\em Kirillov--Kostant--Souriau structure}~\cite{Kirillov-2004}.

Comparing our notions with those established for Poisson manifolds, we observe that for a~function $f \in C^\infty(\g^\star)$, the $\g$-gradient vector field $\nabla^{\cA}(f)$ corresponds to the {\em Hamiltonian vector field $X_f$} for Poisson manifolds, so that the dual distribution $\hh^\cA_\xi$ from equation \eqref{eq:grad-dist} is the Hamiltonian distribution of the Poisson manifold.
It is integrable, as the Hamiltonian vector fields satisfy the identity
\[
[X_f, X_g] = -X_{\{f,g\}}.
\]
The Jacobi identity implies that $l_a = ad_a$ satisfies $[l_a, l_b] = l_{[a,b]}$, so that $\m_\cA$ is closed under the commutator bracket and therefore, $\g(\cA) = \m_\cA \cong \g/\mathfrak{z}(\g)$.
That is, the action of the structure group is induced by the coadjoint action of $G$ on $\g^\star$, and the skew-symmetric non-degenerate bilinear form $\Gcal$ on $\hh^\cA$ from equation \eqref{eq:inner-productA} coincides with the symplectic form on each coadjoint orbit in $\g^\star$.
By the Jacobi identity, this action consists of automorphisms of the Lie algebra structure, whence this symplectic form is preserved under the coadjoint action.

Therefore, the integral leaves of $\m_\cA$ are the coadjoint orbits of $\g^\star$, equipped with their canonical symplectic form, and hence, each orbit is regular in the sense of our definition.

{\bf 2.\ Associative algebras.} The associativity of the product $\bullet$ is equivalent to saying that $l_a l_b  = l_{a\bullet b}$, so that $\{l_a \mid a \in \cA\} \subset \gl(\cA)$ is a subalgebra, which means that the structure algebra~$\g(\cA)$ equals $\m_\cA$ with the Lie bracket being the commutator.

Thus, if we regard $\cA$ as a Lie algebra with the Lie bracket $[a,b] := a \bullet b - b \bullet a$, then the $\G(\cA)$-orbits are the coadjoint orbits on $\cA^\star$, regarded as the dual of a Lie algebra and thus described in the preceding paragraph.

Note that by Proposition \ref{prop:inner-product} the bilinear form $\Gcal$ on these orbits only exists if $\bullet$ is symmetric or antisymmetric.

If $\cA$ is a {\em commutative} and associative algebra, then $\G(\cA)$ and $\g(\cA)$ are {\em abelian} Lie groups, respectively.
In this case, the $\G(\cA)$-orbits of $\cA^\star$ are diffeomorphic to the direct product of a~torus and Euclidean space.

In the two preceding cases, $\m_\cA$ is closed under Lie brackets, so that it coincides with the structure algebra $\g(\cA)$.
This implies that, by the very definition, $\m_\cA$ is integrable having the $\G(\cA)$-orbits in $\cA^\star$ as leaves.
In particular, all orbits are $\m_\cA$-regular.

In contrast, for a {\em Jordan algebra $\Jcal$}, it is no longer true that $\m_\Jcal$ is a Lie algebra, so that not all $\G(\Jcal)$-orbits on the dual $\Jcal^\star$ are $\m_\Jcal$-regular in our sense.
Since the Jordan product is symmetric,
the non-degenerate form $\Gcal$ from (\ref{eq:inner-productA}) defines a pseudo-Riemannian metric on the regular part $\ocal^{\rm reg}_{\m_\Jcal}$ of each orbit.

We shall describe these structures on the $\G(\Jcal)$-orbits on $\Jcal^\star$ and the pseudo-Riemannian metric $\Gcal$ in more details, and we will see how, for some specific type of positive Jordan algebras, and suitable orbits, $\Gcal$ is intimately connected with either the Fisher--Rao metric tensor or with the Bures--Helstrom metric tensor used in classical and quantum information geometry, respectively.
This result strengthen the connection between Jordan algebras and information geometry initially hinted at in \cite{C-J-S-2020-02,C-J-S-2020}.

\section{Jordan algebras and Jordan distributions}\label{sec: Jordan algebras}

Let $\Jcal$ be a real, finite-dimensional Jordan algebra, that is, a real vector space endowed with a~bilinear symmetric product $\{\cdot,\cdot\}$, satisfying for $x,y\in\Jcal$ the {\em Jordan identity}
\[ 
\{\{x, y\}, \{x, x\}\} = \{x, \{y, \{x, x\}\}\}.
\]

By the notions established in the preceding section, we may associate with a Jordan algebra the symmetric bivector field $\cR^\Jcal \in \Gamma\big(\Jcal^\star, S^2(T\Jcal^\star)\big)$ from equation \eqref{eq:R-bivectorfield}, the musical operator $\#_\Jcal\colon T^\ast \Jcal^\star \to T\Jcal^\star$ from equation \eqref{eq:sharp-xi}, the $\Jcal$-dual vector field $\nabla^\Jcal f = \#df \in \mathfrak{X}(\Jcal^\star)$ from~\eqref{eq:grad-A}, and the induced $\Jcal$-dual distribution $\hh^\Jcal \subset T\Jcal^\star$ from equation \eqref{eq:grad-dist}.

In particular, we have $\hh^\Jcal = \dcal^{\m_\Jcal}$ by equation \eqref{eq:grad=dist}, where
\[
\m_\Jcal = \{l_x^\ast \mid x \in \Jcal\} \subset \gl(\Jcal^\star)
\]
is the space of (left-)multiplications with elements $x \in \Jcal$, acting on the dual space $\Jcal^\star$.
For every $\m_\Jcal$-regular point $\xi$, the vector space $\hh^\Jcal_\xi = \dcal^{\m_\Jcal}_\xi$ carries the non-degenerate symmetric bilinear form $\Gcal_\xi$ defined in equation \eqref{eq:inner-productA}.

As we pointed out before, the space $\m_\Jcal$ of (left-)multiplication in $\Jcal$ is not closed under Lie brackets in general.
However, the following is known.

\begin{Lemma}[{cf.\ \cite[Lemma IV.7]{Koecher-1999}}] \label{structure Lie algebra}
 For $x, y \in \Jcal$, the commutator $[l_x, l_y]$ is a derivation of $\Jcal$, and for each $d \in \Der(\Jcal)$
\be \label{eq:Jacobi-g(J)} [d, [l_x, l_y]] = [l_{dx}, l_y] + [l_x, l_{dy}].
\ee
\end{Lemma}

We denote by $\Der_0(\Jcal)$ the span of all elements of the form $[l_x, l_y]$ for $x,y \in \Jcal$.
By equation~\eqref{eq:Jacobi-g(J)}, $\Der_0(\Jcal) \subset \Der(\Jcal)$ is an ideal whose elements are called {\em inner derivations of $\Jcal$}.
This fact can be used to describe the structure Lie algebra of $\Jcal$.

\begin{Definition}\label{prop: structure of the Lie algebra of the structure group}
We define the {\em extended structure Lie algebra $\hat \g(\Jcal)$ of $\Jcal$ }as follows.
As a vector space, $\hat \g(\Jcal)$ is defined by
\[
\hat \g(\Jcal) = \Der_0(\Jcal) \oplus \Jcal.
\]
The Lie bracket on $\hat \g(\Jcal)$ is defined as follows:
\begin{itemize}\itemsep=0pt
\item on $\Der_0(\Jcal) \subset \Der(\Jcal) \subset \gl(\Jcal)$, the Lie bracket is just the commutator between linear maps;
\item for $d \in \Der_0(\Jcal)$ and $x \in \Jcal$, $[d, x] = - [x, d] := d(x) \in \Jcal$;
\item for $x,y \in \Jcal$ we set $[x,y] := [l_x, l_y] \in \Der_0(\Jcal)$.
\end{itemize}
\end{Definition}

In fact, the Jacobi identity for this bracket is easily verified using the definitions and equation~\eqref{eq:Jacobi-g(J)}.
By the definition of this Lie bracket, it follows that ($\hat \g(\Jcal),\Der_0(\Jcal) $) is a transvective symmetric pair in the sense of Definition \ref{df:symm}.

There is a canonical Lie algebra representation of $\hat \g(\Jcal)$ on $\Jcal$, called the {\em standard representation}, defined by
\be \label{eq:can-act}
\phi\colon\ \hat \g(\Jcal) \longrightarrow \gl(\Jcal), \qquad \begin{cases} \phi(d) := d & \text{for $d \in \Der_0(\Jcal) \subset \gl(\Jcal)$},\\ \phi(x) := l_x & \text{for $x \in \Jcal$}. \end{cases}
\ee
Indeed, this defines a Lie algebra homomorphism by the definition of the Lie bracket on $\hat \g(\Jcal)$ and by equation \eqref{eq:der-leftaction}.

Observe that the image $\phi(\hat \g(\Jcal)) \subset \gl(\Jcal)$ is generated by all $l_x$, $x \in \Jcal$, whence equals the structure Lie algebra $\g(\Jcal)$ from Definition \ref{def:structure-Lie}. Thus, there is a surjective Lie group homomorphism $\hat \G(\Jcal) \to \G(\Jcal)$ with differential $\phi$, where $\G(\Jcal) \subset \Gl(\Jcal)$ is the structure group from Definition \ref{def:structure-Lie}.\footnote{Definition \ref{def:structure-Lie} for a Jordan algebra $\Jcal$ coincides with the definition of the structure Lie group and the structure Lie algebra of a Jordan algebra, e.g., in \cite[Chapter IV]{Koecher-1999}.}

In general, $\phi$ may fail to be injective (the kernel of $\phi$ contains the center of $\mathfrak{z}(\Jcal) \subset \Jcal \subset \g(\Jcal)$), so that the structure algebra $\g(\Jcal)$ and the extended structure algebra $\hat \g(\Jcal)$ may not be isomorphic.

Then we obtain the following integrability criterion.

\begin{Proposition} \label{prop:Jordan-integrable}
Let $\Jcal$ be a Jordan algebra.
Then, for the distribution $\hh^\Jcal$ from equation~\eqref{eq:grad-dist}, the following assertions are equivalent.
\begin{enumerate}\itemsep=0pt
\item[$(1)$] $\hh_\xi^\Jcal$ is involutive at $\xi \in \Jcal^\star$,
\item[$(2)$] $\hh_\xi^\Jcal$ is integrable at $\xi \in \Jcal^\star$,
\item[$(3)$] $\Der_0(\Jcal) \cdot \xi \subset \Jcal \cdot \xi$, where the multiplication refers to the dual action of $\Der_0(\Jcal)$, $\Jcal \subset \hat \g(\Jcal)$ on $\Jcal^\star$.
\end{enumerate}
If this is the case, then the maximal integral leaf through $\xi$ is the connected component of $\xi$ in~$\ocal^{\rm reg}_{\m_\Jcal } \subset \ocal = \G(\Jcal) \cdot \xi$.
\end{Proposition}

\begin{proof}
The map $\phi$ from equation \eqref{eq:can-act} defines an action of $\hat\G(\Jcal)$ on $\Jcal^\star$ such that, by equation~\eqref{eq:grad=dist}, it is $\hh_\xi^\Jcal = \dcal_\xi^\Jcal$.
Evidently, $\hat \G(\Jcal) \cdot \xi = \G(\Jcal) \cdot \xi$.
Since ($\hat \g(\Jcal),\Der_0(\Jcal) $) is a transvective symmetric pair, equation \eqref{eq:condition-integrable} is satisfied for $\m := \Jcal \subset \hat \g(\Jcal)$, and the assertion now follows from Corollary \ref{cor:integrable}, as by equation \eqref{eq:Dm} it is
\[
\dcal^{[\Jcal,\Jcal]}_\xi = \dcal^{\Der_0(\Jcal)}_\xi = \Der_0(\Jcal) \cdot \xi, \qquad \dcal^{\Jcal}_\xi = \Jcal \cdot \xi.
\]
\vspace{-13mm}

\end{proof}

\subsection{Jordan frames and the Peirce decomposition}
Let $\Jcal$ be a real finite-dimensional Jordan algebra.
We define the symmetric bilinear form $\tau$ on~$\Jcal$ by
\be \label{eq:innerprod}
\tau(x, y) := \operatorname{tr} l_{\{x, y\}}.
\ee
Observe that for $x,y \in \Jcal$
\[
\tau(gx, gy) = \tau(x,y), \quad g \in \Aut(\Jcal), \qquad \tau(dx,y) + \tau(x, dy) = 0, \quad d \in \Der(\Jcal).
\]
Namely, if $g \in \Aut(\Jcal)$, then $\tau(gx, gy) = \operatorname{tr} l_{g\{x, y\}} = \operatorname{tr} g l_{\{x,y\}} g^{-1} = \tau(x,y)$, and the second identity follows as $\Der(\Jcal)$ is the Lie algebra of $\Aut(\Jcal)$.

A symmetric bilinear form $\beta$ on $\Jcal$ is called {\em associative}, if for all $x,y,z \in \Jcal$
\be \label{eq:l-adjoint}
\beta(\{x, y\}, z) = \beta(x, \{y, z\}),
\ee
i.e., if all $l_x$ are self-adjoint w.r.t.\ $\beta$. Then the following is known.

\begin{Proposition}[{\cite[p.~59]{Koecher-1999}}] \label{prop:tau-associative} The bilinear form $\tau$ from equation \eqref{eq:innerprod} is associative.
\end{Proposition}

An element $c \in \Jcal$ is called an {\em idempotent }if $c^2 := \{c,c\} = c$. Such an idempotent is called {\em primitive}, if there is no decomposition $c = c_1 + c_2$ with idempotents $c_1, c_2 \neq 0$. For an idempotent $c \in \Jcal$, $l_c$ is diagonalizable with eigenvalues in $\big\{0, \frac12, 1\big\}$ \cite[Proposition III.1.2]{F-K-1994}.
Therefore, it follows that
\[ 
\tau(c, c) = \operatorname{tr} l_{\{c, c\}} = \operatorname{tr} l_c \geq 1,
\]
as the trace is the sum of the eigenvalues, and $c$ is in the $1$-eigenspace of $l_c$.

If $\Jcal$ has an identity element $\mathbf{1}_\Jcal$, then a {\em Jordan frame of $\Jcal$ }is a set $(c_i)_{i=1}^r \subset \Jcal$ of primitive idempotents such that
\be \label{eq:Jordan-frame}
\{c_i, c_j\} = \delta_{ij} c_i \qquad \text{and} \qquad c_1 + \dots + c_r = \mathbf{1}_\Jcal.
\ee
Note that
\[
\tau(c_i, c_j) = \tau(\{c_i, c_i\}, c_j) \stackrel{(\ref{eq:l-adjoint})}= \tau(c_i, \{c_i, c_j\}) = 0 \qquad \text{for $i \neq j$},
\]
so that $(c_i)_{i=1}^r$ is an $\tau$-orthogonal system. From this, one can show that the maps $l_{c_i}$ commute pairwise \cite[Lemma IV.1.3]{F-K-1994}.

Let $\cf := \sspan\{c_i\}$. Then, as all $l_{c_i}$ are diagonalizable, there is a $\tau$-orthogonal decomposition of $\Jcal$ into the common eigenspace of $l_{c_i}$, i.e., into spaces of the form
\[
\Jcal_\rho := \{ x \in \Jcal \mid l_{c}(x) = \rho(c) x, c \in \cf\}
\]
for some $\rho \in \cf^\star$. Since $c_i$ has only eigenvalues $\big\{0, \frac12, 1\big\}$ and $\rho({\bf 1}_\Jcal) = 1$, it follows that $\rho = \frac12(\theta_i + \theta_j)$, $i \leq j$, where $(\theta_i)_{i=1}^r \in \cf^\star$ is the dual basis to $(c_i)_{i=1}^r$. That is, we have the $\tau$-orthogonal eigenspace decomposition
\be \label{eq:Peirce}
\Jcal = \bigoplus_{i \leq j} \Jcal_{ij},
\ee
where $\Jcal_{ij} := \Jcal_{\frac12(\theta_i + \theta_j)}$. This is called the {\em Peirce decomposition of $\Jcal$ with respect to the Jordan frame $(c_i)_{i=1}^r$}. For convenience, we let $\Jcal_{ji} := \Jcal_{ij}$ for $i < j$.

\subsection{Semi-simple and positive Jordan algebras}

For a Jordan algebra $\Jcal$, we define the {\em radical of $\Jcal$ }as the null space of $\tau$, i.e.,
\[
{\mathfrak r}(\Jcal) := \{ a \in \Jcal \mid \tau(a,x) = 0 \text{ for all $x \in \Jcal$}\}.
\]
Evidently, ${\mathfrak r}(\Jcal) \subset \Jcal$ is an ideal by Proposition \ref{prop:tau-associative}.

We call $\Jcal$ {\em semi-simple} if ${\mathfrak r}(\Jcal) = 0$, i.e., if $\tau$ is non-degenerate. Moreover, we call $\Jcal$ {\em positive }or {\em formally real}, if $\tau$ is positive definite.

We shall now collect some known results on semi-simple and positive Jordan algebras.

\begin{Proposition} \label{prop:collect-positive-Jordan} Let $\Jcal$ be a semi-simple real Jordan algebra. Then the following hold.
\begin{enumerate}\itemsep=0pt
\item[$(1)$]
$\Jcal$ has a decomposition $\Jcal = \Jcal_1 \oplus \dots \oplus \Jcal_k$ into {\em simple }Jordan algebras $\Jcal_i$, i.e., such that~$\Jcal_i$ does not contain a non-trivial ideal {\rm \cite[Theorem III.11]{Koecher-1999}}.
\item[$(2)$]
$\Jcal$ has an identity element $\mathbf{1}_\Jcal$ {\rm \cite[Theorem III.9]{Koecher-1999}}.
\item[$(3)$] $\Jcal$ is positive iff it admits a positive definite associative bilinear form $\beta$ {\rm \cite[p.~61]{F-K-1994}}.
\item[$(4)$]
If $\Jcal$ is positive, then for every $x \in \Jcal$ there is a Jordan frame $\{c_i\}_{i=1}^r$ with $x \in \sspan(\{c_i\}_{i=1}^r)$ {\rm \cite[Theorem III.1.2]{F-K-1994}}. In particular, $\Jcal$ has Jordan frames.
\item[$(5)$]
If $\Jcal$ is simple and positive and $(c_i)_{i=1}^r$ and $(c_i')_{i=1}^r$ are Jordan frames, then there is an automorphism $h \in \Aut_0(\Jcal)$ with $h(c_i) = c_i'$ for all $i$ {\rm \cite[Theorem IV.2.5]{F-K-1994}},\footnote{In {\rm \cite{F-K-1994}} it is only stated that there exists an element $h \in \Aut(\Jcal)$ with the asserted property; however, $h \in \Aut_0(\Jcal)$ follows from the proof.} where $\Aut_0(\Jcal) \subset \Aut(\Jcal)$ is the identity component. In particular, all Jordan frames have the same number $r$ of elements, and $r$ is called the {\em rank of $\Jcal$}.
\item[$(6)$]
If $\Jcal$ is positive, then for the Peirce spaces in \eqref{eq:Peirce} we have {\rm \cite[Theorem IV.2.1]{F-K-1994}}
\[
\{\Jcal_{ij}, \Jcal_{kl}\} \subset \begin{cases} 0 & \text{if $\{i,j\} \cap \{j,k\} = \varnothing$,}\\
\Jcal_{jl} & \text{if $i=k$, $j \neq l$},\\
\Jcal_{ii} + \Jcal_{jj} & \text{if $\{i,j\} = \{k,l\}$}. \end{cases}
\]
\item[$(7)$]
If $\Jcal$ is positive, then $\Jcal_{ii} = \sspan(c_i)$ is one-dimensional for all $i$.
\item[$(8)$]
{If $\Jcal$ is positive, simple and of rank $r$, then $\Jcal_{jk} \neq 0$ for all $1\leq j\leq k\leq r$ {\rm \cite[Theorem~IV.2.3]{F-K-1994}}.}
\end{enumerate}
\end{Proposition}

\begin{proof}
We only need to show point (7), as it appears not to be explicitly stated in the literature.
Note that $\Jcal_{ii}$ is a subalgebra, as $\{\Jcal_{ii}, \Jcal_{ii}\} \subset \Jcal_{ii}$ by the product relations in point (6), and by definition $c_i = {\bf 1}_{\Jcal_{ii}}$.
Since $\tau|_{\Jcal_{ii}}$ is a positive definite associative bilinear form, it follows from point (3) that $\Jcal_{ii}$ is a positive Jordan algebra as well. However, since $c_i = {\bf 1}_{\Jcal_{ii}}$ is primitive, it follows that each Jordan frame of $\Jcal_{ii}$ consists of $c_i$ only, so that by (4) each $x \in \Jcal_{ii}$ must be a~multiple of $c_i = {\bf 1}_{\Jcal_{ii}}$.
\end{proof}

The fourth of these results is called the {\em spectral theorem of positive Jordan algebras}. It shows that each $x \in \Jcal$ admits a decomposition
\be \label{eq:spectral}
x = \sum_{i=1}^r \lambda_i c_i
\ee
for a Jordan frame $(c_i)_{i=1}^r$, and the decomposition in equation \eqref{eq:spectral} is referred to as the {\em spectral decomposition of $x$}.
The $\lambda_i$'s are called the {\em spectral coefficients of $x$}.
Evidently, the tuple $(\lambda_i)_{i=1}^r$ is defined only up to permutation of the entries.
Furthermore, we call the pair $(n_+, n_-)$, where~$n_+$ and $n_-$ are the number of positive and negative spectral coefficients of $x$ the {\em spectral signature of $x$}.

\begin{Lemma} \label{lem:orbits}
Let $\Jcal$ be a semi-simple, positive Jordan algebra, $(c_i)_{i=1}^r$ a Jordan frame of $\Jcal$ and $x = \sum_i \lambda_i c_i$.
Then, it holds
\begin{gather}
l_x(\Jcal) =\bigg( \bigoplus_{\lambda_a + \lambda_b \neq 0} \Jcal_{ab} \bigg)\oplus \bigg(\bigoplus_{a, \mu} \Jcal_{a\mu}\bigg),\nonumber\\
\Der_0(\Jcal) \cdot x = \bigg(\bigoplus_{\lambda_a - \lambda_b \neq 0} \Jcal_{ab}\bigg) \oplus \bigg(\bigoplus_{a, \mu} \Jcal_{a\mu}\bigg),\nonumber\\
\g(\Jcal) \cdot x = \bigoplus_{a,i} \Jcal_{ai},\label{eq:orbit-xlambda}
\end{gather}
where we use the index convention that $i$, $j$ run over $1, \dots, r$, while $a$, $b$ run over those indices with $\lambda_a \neq 0$, and $\nu$, $\mu$ over those indices with $\lambda_\nu = 0$.
\end{Lemma}

\begin{proof}
By point (7) in Proposition \ref{prop:collect-positive-Jordan}, the Peirce decomposition in equation \eqref{eq:Peirce} reads
\[
\Jcal = \cf \oplus \bigoplus_{a<b} \Jcal_{ab} \oplus \bigoplus_{a, \mu} \Jcal_{a\mu} \oplus \bigoplus_{\mu<\nu} \Jcal_{\mu\nu}, \qquad \cf := \sspan\{c_i\}.
\]
As $l_x(x_{ij}) = \frac12(\lambda_i + \lambda_j) x_{ij}$ for $x_{ij} \in \Jcal_{ij}$, the first equality in equation \eqref{eq:orbit-xlambda} is immediate.

For the second equality, recall that $\Der_0(\Jcal)$ is spanned by $[l_{x_{ij}}, l_{y_{kl}}]$ for $x_{ij}, y_{ij} \in \Jcal_{ij}$, and we compute
\begin{align*}
{}[l_{x_{ij}}, l_{y_{kl}}](x) & = \left\{x_{ij}, \frac12 (\lambda_k + \lambda_l) y_{kl}\right\} - \left\{y_{kl}, \frac12(\lambda_i + \lambda_j) x_{ij}\right\}\\
& = \frac12(\lambda_k + \lambda_l - \lambda_i - \lambda_j) \{x_{ij}, y_{kl}\}.
\end{align*}
Therefore, the relation ``$\subset$'' in the second equality in equation \eqref{eq:orbit-xlambda} follows easily from the bracket relation of the Peirce spaces in point (6) of Proposition \ref{prop:collect-positive-Jordan}.
For the converse inclusion, we compute
\begin{gather*}
 [l_x, l_{x_{ij}}](x)   = \{x, \{x_{ij}, x\}\} - \big\{x_{ij}, x^2\big\}
  = \frac14(\lambda_i + \lambda_j)^2 x_{ij} - \frac12\big(\lambda_i^2 + \lambda_j^2\big) x_{ij} \\
  \hphantom{[l_x, l_{x_{ij}}](x)}{}
  = - \frac14(\lambda_i - \lambda_j)^2 x_{ij}.
\end{gather*}

The third equation then follows as $\g(\Jcal) \cdot x = l_x(\Jcal) + \Der_0(\Jcal) \cdot x$, and $\lambda_a + \lambda_b = \lambda_a - \lambda_b = 0$ cannot both hold for $\lambda_a, \lambda_b \neq 0$.
\end{proof}

\begin{Theorem} \label{thm:orbits}
For a positive simple Jordan algebra $\Jcal$, the following hold:
\begin{enumerate}\itemsep=0pt
\item[$(1)$] The orbits of $\Aut_0(\Jcal)$ are the sets of elements with equal spectral coefficients.
\item[$(2)$] The orbits of the structure group $\G(\Jcal)$ consist of all elements with equal spectral signature.
\end{enumerate}
\end{Theorem}

\begin{proof}
Any automorphism maps (primitive) idempotents to (primitive) idempotents and fixes ${\bf 1}_\Jcal$, whence it maps Jordan frames to Jordan frames. Thus, if $x = \sum_{i=1}^r \lambda_i c_i$ for a Jordan frame~$(c_i)_{i=1}^r$, it follows that for $h \in \Aut_0(\Jcal)$
\[
h(x) = \sum_{i=1}^r \lambda_i h(c_i),
\]
and $(h(c_i))_{i=1}^r$ is again a Jordan frame, so that $x$, $h(x)$ have the same spectral coefficients $(\lambda_i)_{i=1}^r$.

Conversely, if $x$, $y$ have the same spectral coefficients $(\lambda_i)_{i=1}^r$, then
\[
x = \sum_{i=1}^r \lambda_i c_i, \qquad y = \sum_{i=1}^r \lambda_i c_i'
\]
for Jordan frames $(c_i)_{i=1}^r$ and $(c_i')_{i=1}^r$. Thus, by point (5) of Proposition \ref{prop:collect-positive-Jordan}, there is a $h \in \Aut_0(\Jcal)$ with $h(c_i) = c_i'$ and hence, $h(x) = y$. This shows the first statement.

Concerning the second statement, we define the following subsets of $\Jcal$:
\begin{gather*}
\Sigma^m := \{x \in \Jcal \mid \text{$m$ spectral coefficients of $x$ are $\neq 0$}\},\\
\Sigma^{\leq m} := \{x \in \Jcal \mid \text{{\em at most} $m$ spectral coefficients of $x$ are $\neq 0$}\},\\
\Sigma_{n_+, n_-} := \{x \in \Jcal \mid \text{$x$ has spectral signature $(n_+, n_-)$}\}.
\end{gather*}

Evidently,
\be \label{eq:decomp-Sigma_m}
\Sigma^m = \dot \bigcup_{n_++n_- = m} \Sigma_{n_+,n_-}.
\ee
Moreover, by the first assertion, all these sets are $\Aut_0(\Jcal)$-invariant.

There are continuous (in fact, polynomial) functions $a_k\colon \Jcal \to \R$ such that
\be \label{eq:min-pol}
f(x, \lambda) = \lambda^r + \lambda^{r-1} a_{r-1}(x) + \dots + a_0(x)
\ee
is the minimal polynomial of all {\em generic }$x \in \Jcal$, i.e., elements with pairwise distinct spectral coefficients $\lambda_i(x)$ \cite[Proposition II.2.1]{F-K-1994}.
If $x = \sum_i \lambda_i(x) c_i$ is the spectral decomposition of $x$, then $\prod_i(x - \lambda_i(x) x) = 0$ by equation \eqref{eq:Jordan-frame}, and as the roots $\lambda_i(x)$ are pairwise distinct, it follows that
\be \label{eq:min-pol-prod}
f(x, \lambda) = \prod_{i=1}^r (\lambda - \lambda_i(x)),
\ee
and as generic $x$'s are dense in $\Jcal$ \cite[Proposition II.2.1]{F-K-1994}, it follows that equation \eqref{eq:min-pol-prod} holds for {\em any }$x \in \Jcal$. Then $\Sigma^{\leq m}$ are those elements where $\lambda = 0$ is a root of $f(x, \cdot)$ of multiplicity~$\geq r-m$; that is,
\be \label{Sigma-minimal}
\Sigma^{\leq m} = \{x \in \Jcal \mid a_0(x) = \dots = a_{r-m-1}(x) = 0\} \subset \Jcal.
\ee
As the spectral coefficients of $x$ are unchanged under the automorphism group, equation \eqref{eq:min-pol} and equation \eqref{eq:min-pol-prod} imply
\be \label{eq:a_k-invariant}
a_k(h \cdot x) = a_k(x), \qquad x \in \Jcal, \quad h \in \Aut_0(\Jcal).
\ee

We assert that $\Sigma^{\leq m}$ is invariant under $\G(\Jcal)$. For this, fix a Jordan frame $(c_i)_{i=1}^r$ and let~$x = \lambda_a c_a \in \cf \cap \Sigma^m$, using the index summation convention from Lemma \ref{lem:orbits}. Define the map
\[
\Phi_x\colon\ \Aut_0(\Jcal) \times \R^m \longrightarrow \G(\Jcal) \cdot x, \qquad (h, (t_i)) \mapsto h \cdot \exp(l_{t_i c_i}) \cdot x.
\]
Since $l_{t_i c_i}^k x = t_a^k \lambda_a c_a$ by equation \eqref{eq:Jordan-frame}, it follows that
\be \label{eq:action-diag}
\Phi_x(h, (t_i)) = h \cdot \exp(l_{t_i c_i}) \cdot x = h \cdot \big(e^{t_a} \lambda_a c_a\big),
\ee
and as $e^{t_a} \lambda_a c_a \in \Sigma^m$, the $\Aut_0(\Jcal)$-invariance of $\Sigma^m$ implies that $\operatorname{Im}(\Phi_x) \subset \Sigma^m$. Moreover, it follows that the image of the differential $d_{(e,0)} \Phi_x$ is
\[ 
\operatorname{Im} d_{(e,0)} \Phi_x = \sspan\{c_a\} \oplus \Der_0(\Jcal) \cdot x \stackrel{(\ref{eq:orbit-xlambda})}\subset \g(\Jcal) \cdot x = T_x(\G(\Jcal) \cdot x).
\]
In fact, equation \eqref{eq:orbit-xlambda} implies that $\operatorname{Im} d_{(e,0)} \Phi_x = T_x(\G(\Jcal) \cdot x)$ if $x = \lambda_a c_a \in \cf$ is {\em generic in $\Sigma^m$}, that is, if $\lambda_a \neq \lambda_b$ for all $a \neq b$.

This implies that for $x \in \cf \cap \Sigma^m$ generic, there is an open neighborhood $U \subset \G(\Jcal)$ of the identity such that
\[ 
U \cdot x \subset \operatorname{Im}(\Phi_x) \subset \Sigma^m.
\]
Thus, for $X \in \g(\Jcal)$ and $x \in \cf \cap \Sigma^m$ generic, equation \eqref{Sigma-minimal} implies
\be \label{eq:a_k-in-t}
a_k(\exp(tX) \cdot x) = 0, \qquad k = 0, \dots, r-m-1,
\ee
for $|t|$ small enough such that $\exp(tX) \in U$. As all $a_k$ are polynomials, the expressions in equation~\eqref{eq:a_k-in-t} are real analytic in $t$, whence their vanishing for $|t|$ small implies that they vanish for {\em all}~$t \in \R$, in particular for $t=1$. That is, we conclude that
\be \label{eq:a_k-in-1}
a_k(\exp(X) \cdot x) = 0, \qquad k = 0, \dots, r-m-1, \qquad X \in \g(\Jcal)
\ee
for $x \in \cf \cap \Sigma^m$ generic, and taking the closure, it follows that equation \eqref{eq:a_k-in-1} holds for {\em all}~${x \in \cf \cap \Sigma^{\leq m}}$.
Moreover, by the first part, each $x \in \Sigma^{\leq m}$ can be written as $x = h \cdot \tilde x$ for~$\tilde x \in \cf \cap \Sigma^{\leq m}$ and $h \in Aut_0(\Jcal)$. Thus, it holds
\[
a_k(\exp(X) \cdot x) = a_k( h \cdot \operatorname{Ad}_{h^{-1}}(X) \tilde x) \stackrel{(\ref{eq:a_k-invariant})}= a_k(\operatorname{Ad}_{h^{-1}}(X) \cdot \tilde x) \stackrel{(\ref{eq:a_k-in-1})} = 0,
\]
so that equation \eqref{eq:a_k-in-1} holds for all $x \in \Sigma^{\leq m}$ and $X \in \g(\Jcal)$.
Thus, by equation \eqref{Sigma-minimal} it follows that
\[
\exp(\g(\Jcal)) \cdot \Sigma^{\leq m} \subset \Sigma^{\leq m},
\]
and as the connected group $\G(\Jcal)$ is generated by $\exp(\g(\Jcal))$, the asserted $\G(\Jcal)$-invariance of~$\Sigma^{\leq m}$ follows.

Since $\Sigma^m = \Sigma^{\leq m} \backslash \Sigma^{\leq m-1}$ is the difference of two $\G(\Jcal)$-invariant sets, it follows that $\Sigma^m$ is $\G(\Jcal)$-invariant as well.

Next, we assert that $\Sigma_{n_+, n_-} \subset \Sigma^m$ is relatively closed. For if $(x_k)_{k \in \N} \in \Sigma_{n_+, n_-}$ converges to~$x_0 \in \Sigma^m$, then, fixing a Jordan frame $(c_i)_{i=1}^r$, we find $h_k \in \Aut_0(\Jcal)$ such that
\[
y_k := h_k \cdot x_k = \sum_a \lambda_{a,k} c_a, \qquad \lambda_{1,k} \geq \dots \geq \lambda_{m,k}.
\]
Since $y_k \in \Sigma_{n_+,n_-}$ as well, it follows that the signs of $0 \neq \lambda_{a,k}$ are equal for all $k$. As $\Aut_0(\Jcal)$ is compact, we may pass to a subsequence to assume that $h_k \to h_0$, whence $y_k \to h_0 x_0$, i.e.,
\[
h_0 x_0 = \sum_a \lambda_{a, 0} c_a, \qquad \lambda_{a,0} = \lim_{k\to\infty} \lambda_{a,k}.
\]
Since $x_0$ and hence $h_0 x_0 \in \Sigma^m$, it follows that $\lambda_{a,0} \neq 0$ for all $a$, whence $\lambda_{a,0}$ has the same sign as all $\lambda_{a,k}$, so that $h_0 x_0 \in \Sigma_{n_+, n_-}$, i.e., $x_0 \in \Sigma_{n_+, n_-}$.

Thus, equation \eqref{eq:decomp-Sigma_m} is the disjoint decomposition of $\Sigma^m$ into finitely many relatively closed subsets, and since $\G(\Jcal)$ and hence all orbits are connected, it follows that each $\G(\Jcal)$-orbit must be contained in some $\Sigma_{n_+,n_-}$.

On the other hand, as elements with equal spectral coefficients lie in the same $\Aut_0(\Jcal)$-orbit, equation \eqref{eq:action-diag} immediately implies that $\G(\Jcal)$ acts transitively on $\Sigma_{n_+,n_-}$, which completes the proof.
\end{proof}

For a positive Jordan algebra $\Jcal$ we identify $\Jcal$ and $\Jcal^\star$ by the isomorphism
\begin{align*}
&\flat\colon\ \Jcal \longrightarrow \Jcal^\star, \qquad x \longmapsto x^\flat := \tau(x, \cdot),\\
&\#\colon\ \Jcal^\star \longrightarrow \Jcal, \qquad \# := \flat^{-1}.
\end{align*}

By the spectral theorem (cf.\ point (4) of Proposition \ref{prop:collect-positive-Jordan}), for each $\xi \in \Jcal^\star$ there is a Jordan frame $(c_i)_{i=1}^r$ on $\Jcal$ such that
\[
\xi^\# = \lambda_i c_i,\qquad \mbox{ and } \qquad \xi = \lambda_i c_i^\flat,
\]
and we define the spectral coefficients $(\lambda_i)_i$ and the spectral signature $(n_+,n_-)$ of $\xi$ to be the spectral coefficients and signature of $\xi^\# $. We let
\[
\ocal_{n_+,n_-} \subset \Jcal^\star
\]
be the set of elements of spectral signature $(n_+,n_-)$. Furthermore, we define the {\em dual of $\tau$ } to be the scalar product on $\Jcal^\star$ given by
\[ 
\tau^\flat (\eta_1, \eta_2) := \tau\big(\eta_1^\#, \eta_2^\#\big) \qquad \mathrm{or} \qquad \tau^\flat \big(x_1^\flat, x_2^\flat\big) := \tau(x_1, x_2).
\]

For $x,y \in \Jcal$ and $\xi \in \Jcal^\star$, we have $(l_x^\ast \xi)(y) = \xi (l_x y) = \tau\big(\xi^{\#}, l_x y\big) = \tau\big(l_x \xi^{\#}, y\big) = \big(l_x \xi^{\#}\big)^\flat (y)$, so that
\be \label{eq:l*dual}
l_x^\ast \xi = \big(l_x \xi^{\#}\big)^\flat.
\ee
Therefore, by the definition of the dual action,
it follows that
\be \label{eq:dual-orbit}
\Aut_0(\Jcal) \cdot \xi = \big(\Aut_0(\Jcal) \cdot \xi^{\#}\big)^\flat, \qquad \G(\Jcal) \cdot \xi = \big(\G(\Jcal) \cdot \xi^{\#}\big)^\flat,
\ee
so that, by Theorem \ref{thm:orbits}, we obtain that the orbits of the action of $\G(\Jcal)$ on $\Jcal^\star$ are the sets~$\ocal_{n_+,n_-}$.

The {\em open cone of squares in $\Jcal$ }is
\[
\Omega_\Jcal := \mathrm{Int} \big\{x^2 \mid x \in \Jcal\big\}.
\]
Looking at the spectral decomposition in equation \eqref{eq:spectral}, it follows that $x \in \Omega_\Jcal$ iff all its spectral coefficients are positive iff $l_x$ is positive definite, and the latter description shows that~$\Omega_\Jcal$ is indeed a convex cone; in fact, it easily follows from this characterization that
\begin{gather} \Omega_\Jcal = \ocal_{r,0} = \g(\Jcal) \cdot {\bf 1}_\Jcal,\qquad
\overline{\Omega}_\Jcal = \dot \bigcup_{n_+ \geq 0} \ocal_{n_+,0}.\label{eq:positive-orbits}
\end{gather}

\begin{Theorem} \label{thm:main}
Let $\Jcal$ be a positive, simple Jordan algebra with structure group $G(\Jcal) \subset \Gl(\Jcal)$.
Then $\xi \in \Jcal^\star$ is $\m_\Jcal$-regular iff the spectral coefficients $(\lambda_i)$ of $\xi$ satisfy:
\be \label{eq:form-regular}
\lambda_a + \lambda_b \neq 0 \qquad \text{whenever $\lambda_a, \lambda_b \neq 0$}.
\ee
In particular, the $\G(\Jcal)$-orbit $\ocal_{n_+,n_-}$ is $\m_\Jcal$-regular iff $n_+ = 0$ or $n_-=0$, i.e., iff it is contained in $\overline{\Omega}_\Jcal$ or $-\overline{\Omega}_\Jcal$.
\end{Theorem}

\begin{proof}
Let $x := \xi^\# \in \Jcal$.
By equation \eqref{eq:dual-orbit}, $\Aut_0(\Jcal) \cdot \xi \subset \Jcal \cdot \xi$ iff $\Aut_0(\Jcal) \cdot x \subset \Jcal \cdot x = l_x(\Jcal)$, and, {recalling point (8) in Proposition \ref{prop:collect-positive-Jordan}}, by equation~\eqref{eq:orbit-xlambda} this condition is satisfied iff equation~\eqref{eq:form-regular} holds.
Recalling Proposition \ref{prop:Jordan-integrable}, the first statement follows.

If $n_+, n_- > 0$, then evidently, $\ocal_{n_+, n_-}$ contains elements two of whose spectral coefficients satisfy $\lambda_a = - \lambda_b \neq 0$, so that $\ocal_{n_+, n_-}$ is not $\m_\Jcal$-regular.

On the other hand, on $\ocal_{n_+, 0}$ ($\ocal_{0, n_-}$, respectively) $\lambda_a, \lambda_b > 0$ ($< 0$, respectively) so that equation \eqref{eq:form-regular} holds; whence $\ocal_{n_+, 0}$ and $\ocal_{0, n_-}$ are the only $\m_\Jcal$-regular orbits, and by equation~\eqref{eq:positive-orbits} these are the orbits contained in $\bar \Omega_\Jcal$ or $-\bar \Omega_\Jcal$, respectively.
\end{proof}

Let us now describe the pseudo-Riemannian metric $\Gcal$ on $\ocal^{\rm reg}_{\m_{\Jcal}}$.
Take $\xi\in\Jcal^{\star}$ with spectral decomposition
\be \label{eq:define-xi}
\xi = \lambda_a c_a^\flat \in \Jcal^\star \quad \Rightarrow \quad x := \xi^\# = \lambda_a c_a \in \Jcal
\ee
for some Jordan frame $(c_i)_{i=1}^r$, and assume it satisfies equation \eqref{eq:form-regular}.
Then, it holds
\[
T_\xi \ocal = \mathfrak{g}(\Jcal) \cdot \xi = (\mathfrak{g}(\Jcal) \cdot x)^\flat \stackrel{(\ref{eq:orbit-xlambda})}= \bigoplus_{a,i} \Jcal_{ai}^\flat,
\]
and we have the following Proposition.

\begin{Proposition} \label{prop:metric-Jordan}
Let $\xi = \lambda_a c_a^\flat \in \Jcal^\star$ be as above.
Then, it holds
\be \label{eq:G-xi}
\Gcal_\xi = \sum_{a,i} \frac2{\lambda_a + \lambda_i} \tau^\flat|_{\Jcal_{ai}^\flat},
\ee
which is equivalent to
\[
\Gcal_\xi\big(x_{ai}^\flat, y_{bj}^\flat\big) = \begin{cases} \dfrac2{\lambda_a + \lambda_i} \tau^\flat\big(x_{ai}^\flat, y_{ai}^\flat\big) & \text{if $(a,i) = (b,j)$,}\\[1mm] 0 & \text{else.} \end{cases}
\]
\end{Proposition}

\begin{proof}
For $x_{ij} \in \Jcal_{ij}$, we have
\[
l_{x_{ij}}^\ast(\xi) \stackrel{(\ref{eq:l*dual})}= \big(l_{x_{ij}}\xi^\#\big)^\flat = \frac12 (\lambda_i + \lambda_j) x_{ij}^\flat.
\]
Therefore, evaluating both sides of equation \eqref{eq:inner-productA} gives us
\begin{align*}
\Gcal_\xi(l_{x_{ai}}^\ast(\xi), l_{y_{bj}}^\ast(\xi)) & = \frac14 (\lambda_a + \lambda_i)(\lambda_b + \lambda_j) \Gcal_\xi\big(x_{ai}^\flat, y_{bj}^\flat\big)\\
\Gcal_\xi(l_{x_{ai}}^\ast(\xi), l_{y_{bj}}^\ast(\xi)) & = \xi(\{x_{ai}, y_{bj}\}) = \tau\big(\xi^\#, \{x_{ai}, y_{bj}\}\big)\\
& = \tau(\{\xi^\#, x_{ai}\}, y_{bj}) \stackrel{(\ref{eq:define-xi})} = \frac12 (\lambda_a + \lambda_i) \tau(x_{ai}, y_{bj}).
\end{align*}
Since both equations must be equal, \eqref{eq:G-xi} follows as the Peirce decomposition ${\Jcal =  \bigoplus_{ij} \Jcal_{ij}}$ is $\tau$-orthogonal.
\end{proof}

\begin{Remark}\quad
\begin{enumerate}\itemsep=0pt
\item[(1)]
Comparing the description of the regular points in $\ocal$ in Theorem \ref{thm:main} and equation \eqref{eq:G-xi}, it follows that $\Gcal_\xi$ has a pole of order $1$ on $\ocal \backslash \ocal^{\rm reg}_{\m_{\Jcal}}$.

\item[(2)]
As $\Gcal_\xi$ is positive or negative definite on $\Jcal_{aa}^\flat$, depending on the sign of $\lambda_a \neq 0$, it follows that $\Gcal_\xi$ is indefinite at any regular point of spectral signature $(n_+, n_-)$ with $n_+, n_- > 0$.

That is, $\Gcal$ on $\ocal$ is definite (and hence defines a Riemannian metric) iff $\ocal = \G(\Jcal) \cdot \xi$ is a regular orbit, iff $\ocal \subset \overline{\Omega}_\Jcal$ is contained in the closure of the cone of squares ($\Gcal > 0$) or $\ocal \subset -\overline{\Omega}_\Jcal$ ($\Gcal < 0$).

\item[(3)]
It is also evident from the description of regular points in Theorem \ref{thm:main} that for a non-regular orbit $\ocal_{n_+,n_-}$ with $n_+,n_- > 0$ the regular part $\left(\ocal_{n_+,n_-}\right)^{\rm reg}_{\m_{\Jcal}}$ is not path connected.

\item[(4)]
Note that the Riemannian metric $\Gcal$ on $\ocal_{n_+,0}$ (and, similarly, $-\Gcal$ on $\ocal_{0,n_-}$) is not complete.
Namely, for $t>0$, the curve
\[
\alpha(t) := t^2 (c_1 + \dots + c_{n_+})^\flat \in \ocal_{n_+,0}
\]
for a Jordan frame $(c_i)_{i=1}^r$ has constant speed with respect to $\Gcal$ because
\[
\Gcal_{\alpha(t)}(\dot \alpha, \dot \alpha) = \sum_{a=1}^{n_+} \frac1{t^2} \tau(2t c_a, 2t c_a) = 4 \sum_{a=1}^{n_+} \tau(c_a, c_a).
\]
However, $\alpha$ cannot be extended in $\ocal_{n_+,0}$ at $t = 0$.
\end{enumerate}
\end{Remark}

\subsection{Examples} \label{subsec: examples}

We shall now describe the metric $\Gcal$ for the standard examples of positive Jordan algebras.

\subsection{The Fisher--Rao metric for finite sample spaces}\label{Ex: F-R}

We regard $\Jcal := \R^n $ {as a positive Jordan algebra whose algebraic operations are defined in a~component-wise way.
Then, it is not difficult to see that $\Omega_\Jcal$ can be identified with the first orthant $\R^n_+ \subset \R^{n}\cong \Jcal^\star$.}
The metric $\Gcal_\xi$ at $\xi = (\xi_1, \dots, \xi_n) \in \Omega_\Jcal$ is given by
\[
\Gcal_\xi(u, v) = \sum_i \dfrac1{\xi_i} u_i v_i, \qquad u = (u_i)_{i=1}^n, v = (v_i)_{i=1}^n \in \R^n.
\]
{When interpreting $\Omega_\Jcal$ as the set of positive finite measures on $\mathcal{X}_{n} = \{1, \dots, n\}$, it is clear that~$\Gcal$ {is such its pullback to the submanifold of strictly positive probability distributions on~$\mathcal{X}_{n}$ (i.e., open interior of the unit simplex inside $\R^n_+$) coincides with the {\em Fisher--Rao metric tensor }which naturally occurs in classical information geometry \cite{Amari-1985,A-J-L-S-2017}.}
As partially noted in \cite{C-J-S-2020-02,C-J-S-2020}, this instance shows that we may look at the non-normalized Fisher--Rao metric tensor on $\Omega_\Jcal = \R^n_+$ as the analogue of the homogeneous symplectic form on co-adjoint orbits in the case of Lie algebras.}

\subsection[The Jordan algebras M\_n\^{}\{sa\}(K), K = R, C, bH]{The Jordan algebras $\boldsymbol{ M_n^{\rm sa}(\K)}$, $\boldsymbol{\K = \R, \C, \bH}$}

Let $\K$ denote either the real, complex or quaternionic numbers, and we define the Jordan algebra of self-adjoint matrices
\[
M_n^{\rm sa}(\K) := \{ A \in \K^{n \times n} \mid A = A^\ast\}, \qquad \{A, B\} := \dfrac12 (AB + BA).
\]
For convenience, we replace $\tau$ from equation \eqref{eq:innerprod} by the associative inner product
\[
\hat \tau(A, B) := \operatorname{Tr}(AB),
\]
so that $\tau$ and $\hat \tau$ only differ by the multiplicative constant $\frac 1n \dim_\R M_n^{\rm sa}(\K)$.

Let $E_{ij} \in \K^{n\times n}$ denote the matrix with a $1$ in the $(i,j)$-entry.
Then $\{E_{11}, \dots, E_{nn}\}$ is a~Jordan frame of $M_n^{\rm sa}(\K)$, and the remaining Pierce spaces with respect to this frame are given as
\[
(M_n^{\rm sa}(\K))_{ij} = \{ z E_{ij} + \bar z E_{ji} \mid z \in \K\}, \qquad i < j.
\]
For $\K = \R, \C$ and $\bH$, the automorphism group of $M_n^{\rm sa}(\K)$ is ${\rm SO}(n)$, ${\rm U}(n)$ and ${\rm Sp}(n)$, respectively, acting on $M_n^{\rm sa}(\K)$ by conjugation. Thus, in particular, each $A \in M_n^{\rm sa}(\K)$ is diagonalizable by an element in the automorphism group, so that the spectral coefficients are the eigenvalues of~$A$. Thus, by Proposition \ref{prop:metric-Jordan}, for $\xi = \lambda_a E_{aa}^\flat \in (M_n^{\rm sa}(\K))^\star$ the metric~$\Gcal_\xi$ reads
\[
\Gcal_\xi\big(z E_{ai}^\flat, w E_{bj}^\flat\big) = \begin{cases} \dfrac2{\lambda_a + \lambda_i} (z\bar w + w \bar z) & \text{if $(a,i) = (b,j)$,}\\[1mm] 0 & \text{else.} \end{cases}
\]
{Moreover, because of Proposition \ref{prop:inner-product}, it follows that $\Gcal$ is preserved by the automorphism group of $M_n^{\rm sa}(\K)$.}

As already mentioned in the introduction, and in accordance with the results put forward in~\cite{C-J-S-2020-02,C-J-S-2020}, an interesting link between Jordan algebras and quantum information geometry appears when $\mathbb{K}=\mathbb{C}$.
In this case, we may identify $M_n^{\rm sa}(\C)$ with the Jordan algebra of self-adjoint observables of a finite-level quantum system with Hilbert space $\mathcal{H}\cong\C^{n}$.
Then, if we focus on the $\m_{\Jcal}$-regular orbit $\Omega_{\Jcal}$ of faithful, non-normalized quantum states, which can be identified with the dual of the orbit of invertible positive matrices in $M_n^{\rm sa}(\C)$, the metric tensor $\Gcal$ is such that its pullback to the submanifold of faithful quantum states, determined by the condition $\operatorname{Tr}(A)=1$, coincides with the so-called Bures--Helstrom metric tensor \cite{B-Z-2006, Dittmann-1993, Dittmann-1995, Helstrom-1967, Helstrom-1968, Helstrom-1969, Helstrom-1976, Safranek-2017, Safranek-2018, S-A-G-P-2020, Uhlmann-1992}.
Analogously, if we focus on the $\m_{\Jcal}$-regular orbit through non-normalized pure states, which are identified with rank-one matrices in $M_n^{\rm sa}(\C)$, the metric tensor $\Gcal$ is such that its pullback to the submanifold of pure states, determined by the condition $\operatorname{Tr}(A)=1$, is a multiple of the Fubini--Study metric tensor, essentially because of its unitary invariance.
Accordingly, and in analogy with the Fisher--Rao metric tensor seen before, we may think of the non-normalized version of the Bures--Helstrom metric tensor and of the Fubini--Study metric tensor as the analogue of the Kostant--Kirillov--Souriau symplectic form in the case of the Jordan algebra~$M_n^{\rm sa}(\C)$.\looseness=1

\subsection[The spin-factor Jordan algebra Jspin(n)]{The spin-factor Jordan algebra $\boldsymbol{\Jspin(n)}$}

Denoting the standard inner product of $\R^n$ by $\langle \cdot, \cdot \rangle$, we let
\[
\Jspin(n) := \R^{n+1} = \R {\bf 1} \oplus \R^n, \qquad \{x,y\} := \langle x,y\rangle {\bf 1}, \qquad x,y \in \R^n,
\]
and where ${\bf 1}$ is the identity element of $\Jspin(n)$. An associative inner product is given by
\[
\hat \tau|_{\R^n} = \langle \cdot, \cdot \rangle, \qquad \hat \tau({\bf 1}, {\bf 1}) := 1, \qquad \hat \tau({\bf 1}, \R^n) = 0.
\]
The automorphism group is ${\rm SO}(n)$, acting on $\R^n$ and fixing ${\bf 1}$. Every Jordan frame is given by
\[
\left\{ \dfrac12({\bf 1} + e_0), \dfrac12 ({\bf 1} - e_0)\right\}
\]
for a fixed unit vector $e_0 \in \R^n$, and the Peirce space complementary to the Jordan frame is
\[
\Jspin(n)_{12} := e_0^\perp.
\]
The two spectral coefficients of an element $X = t {\bf 1} + x$ are
\[
\lambda_1 = \frac 12 (t+ \|x\|), \qquad \lambda_2 = \frac 12 (t- \|x\|),
\]
where $\|\cdot\|$ denotes the norm on $\R^n$ induced by $\langle \cdot, \cdot \rangle$. Therefore, $\xi \in \Jspin(n)^\star$ is regular iff~$0 \neq \lambda_1 + \lambda_2 = \hat \tau^\flat\big(\xi, {\bf 1}^\flat\big)$. If
\[
\xi = t_0 {\bf 1}^\flat + s_0 e_0^\flat \in \Jspin(n)^\star, \qquad t_0 \neq 0
\]
is regular, where $e_0 \in \R^n$ is a unit vector, then the tangent vectors $X_1, X_2 \in T_\xi \ocal$ are of the form
\[
X_i = t_i {\bf 1}^\flat + s_i e_0^\flat + x_i^\flat,
\]
where $x_i \in e_0^\perp$, and where $t_0 = \pm s_0 \Rightarrow t_i = \pm s_i$. The spectral coefficients of $\xi$ are $\lambda_1 = \frac12(t_0+s_0)$ and $\lambda_2 = \frac12(t_0-s_0)$, and
\[
X_i = (t_i+s_i) \frac12({\bf 1} + e_0)^\flat + (t_i-s_i) \frac12({\bf 1} - e_0)^\flat + x_i^\flat.
\]
Therefore,
\begin{align*}
\Gcal_\xi(X_1, X_2) = \dfrac 2{t_0 + s_0} (t_1 + s_1)(t_2 + s_2) + \dfrac 2{t_0 - s_0} (t_1 - s_1)(t_2 - s_2) + \dfrac2{t_0} \langle x_1, x_2 \rangle.
\end{align*}
This metric is positive definite if $t_0 \geq |s_0|$ and negative definite if $t_0 \leq -|s_0|$, as predicted by Proposition \ref{prop:metric-Jordan}.

\begin{Remark}
According to the classification given in \cite[Theorem V.3.7]{F-K-1994}, the first two classes of examples discussed above give a complete list of simple, positive Jordan algebras up to the {\em Albert algebra}, a $27$-dimensional simple Jordan algebra of rank $3$. Its automorphism group is $F_4$ and the structure algebra is of type $E_6$.

While it would be possible but elaborate to calculate the regular points and the inner product~$\Gcal$ on the tangent to the orbit at a regular point, our results in Theorem \ref{thm:main} and Proposition~\ref{prop:metric-Jordan}, allow understanding the structure without these explicit calculations.
\end{Remark}

\subsection{Non-simple, semisimple positive Jordan algebras}

By point (1) of Proposition \ref{prop:collect-positive-Jordan}, each positive, semisimple Jordan algebra admits a~decomposition~$\Jcal = \Jcal_1 \oplus \dots \oplus \Jcal_k$ into positive {\em simple }Jordan algebras, so that both the automorphism and the structure group of $\Jcal$ are the direct sum of the automorphism and structure group of the simple factors $\Jcal_i$, respectively.
Then, applying Theorems \ref{thm:orbits} and \ref{thm:main}, and Proposition \ref{prop:metric-Jordan}, it follows that the $\G(\Jcal)$-orbits are of the form
\[
\ocal^1_{n^1_+, n^1_-} \times \dots \times \ocal^k_{n^k_+, n^k_-} \subset \Jcal^\star = \Jcal_1^\star \oplus \dots \oplus \Jcal^\star_k,
\]
where $\ocal^i_{n^i_+, n^i_-} \subset \Jcal^\star_i$ are $\G(\Jcal_i)$-orbits. In particular, such an orbit is regular iff $n^i_+ n^i_- = 0$ for all~$i$, and the metric $\Gcal$ on the regular part of this orbit is given by (\ref{eq:G-xi}).

\section{Discussion} \label{sec: conclusions}

As mentioned in Section~\ref{sec: introduction}, the mathematics of Jordan algebras has been intensively studied, and we want to discuss here where our investigation place itself in this context.

First of all, it is clear that our construction is highly influenced and inspired by Kirillov's theory, and the relation between Jordan algebras and Kirillov's theory has been already investigated in the literature \cite{H-N-O-1995,H-N-O-1996,H-N-O-1996-2}.
However, the approaches and constructions already available are different from ours, both conceptually and in terms of mathematical structures.
For instance, in the works mentioned above, the emphasis is on the study of nilpotent coadjoint orbits of convex types, while we do not study the coadjoint orbits of the structure group $\G(\Jcal)$, rather we try to develop an analogue of coadjoint orbits for a Jordan algebra.
In order to do so, it turns out we need to rely on the structure group $\G(\Jcal)$ and on (some of) its homogeneous spaces.
A clear departure point from ordinary Kirillov theory is then the fact that the orbits we find do not possess a natural symplectic structure, but rather a pseudo-Riemannian one.
The shift from an antisymmetric tensor to a symmetric one is clearly related to the fact that we develop Kirillov's theory for Jordan algebras, which have a symmetric product.
Moreover, as commented in Sections~\ref{sec: introduction} and~\ref{subsec: examples}, a direct output of our investigation is the fact that some relevant Riemannian structures that naturally arise in the context of Classical and quantum information geometry (namely, the Fisher--Rao metric tensor and the Bures--Helstrom metric tensor) can be thought of as the Jordan-algebraic analogue of the natural homogeneous symplectic form on the co-adjoint orbits of a Lie group.

Other important avenues of research around the mathematics of Jordan algebras concern the~relation between symmetric spaces/cones/domains and Jordan algebras, as well as the differential geometric structures related with the algebraic structure of Jordan algebras \cite{Bertram-2000,B-N-2005,Chu-2012,F-K-1994,Koecher-1999,Upmeier-1985}.

In all these investigations, the geometrical objects are subsets of a given Jordan algebra $\Jcal$, while we consider positive linear functionals on $\Jcal$.
This change of perspective is dictated by the role Jordan algebras play in information geometry, where the relevant objects live in the dual space of $\Jcal$ (e.g., classical probability distributions, quantum states).
It is true that it is possible to identify $\Jcal$ with its dual in finite dimensions, and we exploit this technicality in our proofs, but the focus of our approach is conceptually different from the previous ones, and this difference is particularly relevant with respect to possible generalizations to infinite dimensions (a task that we plan to address in the future).

Again in relation with the above-mentioned references, the symmetric cones in real Hilbert spaces that turn out to be in one-to-one correspondence with formally real (Euclidean) Jordan algebras, in the sense that every such symmetric cone can be realized as the open cone $\Omega_\Jcal$ of invertible positive elements of a suitable Jordan algebra $\Jcal$, are given from the onset.
From our point of view, the open cone $\Omega_\Jcal$ is not an {\itshape a priori} datum of the problem, but it appears because it is diffeomorphic to an integral manifold of the Jordan distribution $\hh^\Jcal$.
Therefore, $\Omega_\Jcal$ is a byproduct of our investigation of the integrability properties of the Jordan distribution canonically associated with the Jordan algebra product.
Moreover, the symmetric cones are also naturally endowed with a Riemannian structure $g$, which is invariant under the action of the symmetry group $\G(\Jcal)$ of the cone itself, while the Riemannian structure we obtain is invariant only under the subgroup of $\G(\Jcal)$ composed by automorphisms of $\Jcal$.
This instance follows from the fact that we obtain the Riemannian structure by implementing a symmetric analogue of Kirillov theory, and not by symmetry considerations.
In particular, in the case of the real associative Jordan algebra $\mathbb{R}^{n}$, the $\G(\Jcal)$-invariant Riemannian metric tensor $g$ on $\Omega_\Jcal=\mathbb{R}^{n}_{+}$ can be written as \cite[Theorem~2.3.19]{Chu-2012}
\[
g=\sum_{j=1}^{n}\frac{\mathrm{d}p^{j} \otimes \mathrm{d}p^{j}}{(p^{j})^{2}},
\]
where $\{p^{j}\}_{j=1,\dots,n}$ is a Cartesian coordinate system adapted to $\Omega_\Jcal$ in the sense that elements in~$\Omega_\Jcal$ have strictly positive values of the coordinates.
This Riemannian metric tensor is different from the Fisher--Rao metric tensor appearing in classical information geometry \cite{A-J-L-S-2017}.
However, as shown in Section~\ref{subsec: examples} (in accordance with \cite{C-J-S-2020}), the Riemannian structure emerging from our construction is precisely the Fisher--Rao metric tensor.
Moreover, the $\G(\Jcal)$ invariance of~$g$ implies it is invariant with respect to dilations also when $\Jcal = \mathcal{B}_{\rm sa}(\mathbb{C}^{n})\cong M_{n}^{\rm sa}(\mathbb{C})$, so that $g$ can not be the Bures--Helstrom metric tensor appearing in quantum information geometry \cite{Dittmann-1993}, while, again referring to Section~\ref{subsec: examples}, the Riemannian metric tensor we obtain reduces to the Bures--Helstrom metric tensor when $\Jcal = \mathcal{B}_{\rm sa}(\mathbb{C}^{n})\cong M_{n}^{\rm sa}(\mathbb{C})$.

Another relevant point to stress is that our construction relies entirely on the algebraic Jordan product.
This detail contributes to further differentiate our work from previous ones, where part of the geometrical structures associated with Jordan algebras necessarily involve the so-called \emph{Jordan triple product}, a trilinear map that exists in every Jordan algebra (and also in more general objects known as \emph{Jordan triple systems}).
For instance, the $\G(\Jcal)$-invariant Riemannian metric tensor discussed above is defined in terms of the Jordan triple product of~$\Jcal$ \cite[Theorem~2.3.19]{Chu-2012}, while our Riemannian metric tensor only requires the Jordan product.
Also, the Jordan triple product induces a natural Lie triple system that leads to a curvature-like tensor~\cite{Bertram-2000}, and again, this differential geometric object makes use of the Jordan product only indirectly, highlighting the difference with our approach.
At this point, it is worth noting that we plan to address the role of the Jordan triple product and its associated Lie triple system in information geometry because we believe they are connected with the so-called Amari--Cencov tensor in the classical case, and with a non-symmetric generalization of that tensor in the quantum case.\looseness=1

Therefore, while and perhaps because our approach is partly extrinsically motivated by classical and quantum information geometry, and partly intrinsically by developing an orbit method-like theory, our constructions and results are different from previous ones and, taken together, they yield a fuller picture of the structure of Jordan algebras.

\subsection*{Acknowledgements}

F.M.C.\ acknowledges that this work has been supported by the Madrid Government (Comunidad de Madrid-Spain) under the Multiannual Agreement with UC3M in the line of ``Research Funds for Beatriz Galindo Fellowships'' (C\&QIG-BG-CM-UC3M), and in the context of the V PRICIT (Regional Program of Research and Technological Innovation).
He also wants to thank the incredible support of the Max Planck Institute for the Mathematics in the Sciences in Leipzig, where he was formerly employed when this work was initially started and developed.
L.S.~acknowledges partial support by grant SCHW893/5-1 of the
Deutsche Forschungsgemeinschaft, and also expresses his gratitude for the hospitality of the Max Planck Institute for the Mathematics in the Sciences in Leipzig during numerous visits.
This publication is based upon work from COST Action CaLISTA CA21109 supported by COST (European Cooperation in Science and Technology, www.cost.eu).
We also thank the anonymous referees for their valuable comments and suggestions which enabled us to significantly improve this manuscript.

\pdfbookmark[1]{References}{ref}
\LastPageEnding


\begin{thebibliography}{99}
\footnotesize\itemsep=0pt

\bibitem{A-S-2001}
Alfsen E.M., Shultz F.W., State spaces of operator algebras. {B}asic theory,
 orientations, and {$C^*$}-products, \textit{Math. Theory Appl.}, \href{https://doi.org/10.1007/978-1-4612-0147-2}{Birkh\"auser}, Boston,
 MA, 2001.

\bibitem{Amari-1985}
Amari S.-I., Differential-geometrical methods in statistics, \textit{Lect. Notes
 Stat.}, Vol.~28, \href{https://doi.org/10.1007/978-1-4612-5056-2}{Springer}, Berlin, 1985.

\bibitem{Amari-2016}
Amari S.-I., Information geometry and its applications, \textit{Appl. Math.
 Sci.}, Vol. 194, \href{https://doi.org/10.1007/978-4-431-55978-8}{Springer}, Tokyo, 2016.

\bibitem{A-N-2000}
Amari S.-I., Nagaoka H., Methods of information geometry, \textit{Transl. Math.
 Monogr.}, Vol. 191, \href{https://doi.org/10.1090/mmono/191}{American Mathematical Society}, Providence, RI, 2000.

\bibitem{A-J-L-S-2015}
Ay N., Jost J., L\^e H.V., Schwachh\"ofer L., Information geometry and
 sufficient statistics, \href{https://doi.org/10.1007/s00440-014-0574-8}{\textit{Probab. Theory Related Fields}} \textbf{162}
 (2015), 327--364, \href{https://arxiv.org/abs/1207.6736}{arXiv:1207.6736}.

\bibitem{A-J-L-S-2017}
Ay N., Jost J., L\^e H.V., Schwachh\"ofer L., Information geometry,
 \textit{Ergeb. Math. Grenzgeb.~(3)}, Vol.~64, \href{https://doi.org/10.1007/978-3-319-56478-4}{Springer}, Cham, 2017.

\bibitem{Baez-2022}
Baez J.C., Getting to the bottom of Noether's theorem, in The Philosophy and
 Physics of Noether's Theorems: a Centenary Volume, \href{https://doi.org/10.1017/9781108665445.005}{Cambridge University Press}, Cambridge, 2022, 66--99,
 \href{https://arxiv.org/abs/2006.14741}{arXiv:2006.14741}.

\bibitem{B-B-M-2016}
Bauer M., Bruveris M., Michor P.W., Uniqueness of the {F}isher--{R}ao metric on
 the space of smooth densities, \href{https://doi.org/10.1112/blms/bdw020}{\textit{Bull. Lond. Math. Soc.}} \textbf{48}
 (2016), 499--506, \href{https://arxiv.org/abs/1411.5577}{arXiv:1411.5577}.

\bibitem{B-Z-2006}
Bengtsson I., Zyczkowski K., Geometry of quantum states: an introduction to
 quantum entanglement, \href{https://doi.org/10.1017/cbo9780511535048}{Cambridge University Press}, Cambridge, 2006.

\bibitem{Berthier-2020}
Berthier M., Geometry of color perception. {P}art~2: perceived colors from real
 quantum states and {H}ering's rebit, \href{https://doi.org/10.1186/s13408-020-00092-x}{\textit{J.~Math. Neurosci.}} \textbf{10}
 (2020), 14, 25~pages, \href{https://arxiv.org/abs/hal-02342456}{arXiv:hal-02342456}.

\bibitem{B-P-P-2022}
Berthier M., Prencipe N., Provenzi E., A quantum information-based refoundation
 of color perception concepts, \href{https://doi.org/10.1137/22M1476071}{\textit{SIAM J.~Imaging Sci.}} \textbf{15}
 (2022), 1944--1976.

\bibitem{B-P-2022}
Berthier M., Provenzi E., Quantum measurement and colour perception: theory and
 applications, \href{https://doi.org/10.1098/rspa.2021.0508}{\textit{Proc.~A.}} \textbf{478} (2022), 20210508, 25~pages,
 \href{https://arxiv.org/abs/hal-03268152}{arXiv:hal-03268152}.

\bibitem{Bertram-2000}
Bertram W., The geometry of {J}ordan and {L}ie structures, \textit{Lecture
 Notes in Math.}, Vol. 1754, \href{https://doi.org/10.1007/b76884}{Springer}, Berlin, 2000.

\bibitem{B-N-2005}
Bertram W., Neeb K.-H., Projective completions of Jordan pairs, {P}art {II}:
 {M}anifold structures and symmetric spaces, \href{https://doi.org/10.1007/s10711-004-4197-6}{\textit{Geom. Dedicata}}
 \textbf{112} (2005), 73--113, \href{https://arxiv.org/abs/math.GR/0401236}{arXiv:math.GR/0401236}.

\bibitem{Bures-1969}
Bures D., An extension of {K}akutani's theorem on infinite product measures to
 the tensor product of semifinite {$w^{\ast}$}-algebras, \href{https://doi.org/10.2307/1995012}{\textit{Trans. Amer.
 Math. Soc.}} \textbf{135} (1969), 199--212.

\bibitem{Cencov-1982}
\v{C}encov N.N., Statistical decision rules and optimal inference,
 \textit{Transl. Math. Monogr.}, Vol.~53, \href{https://doi.org/10.1090/mmono/053}{American Mathematical Society},
 Providence, RI, 1982.

\bibitem{C-C-I-M-V-2019}
Chru\'sci\'nski D., Ciaglia F.M., Ibort A., Marmo G., Ventriglia F., Stratified
 manifold of quantum states, actions of the complex special linear group,
 \href{https://doi.org/10.1016/j.aop.2018.11.015}{\textit{Ann. Physics}} \textbf{400} (2019), 221--245, \href{https://arxiv.org/abs/1811.07406}{arXiv:1811.07406}.

\bibitem{Chu-2012}
Chu C.-H., Jordan structures in geometry and analysis, \textit{Cambridge Tracts
 in Math.}, Vol. 190, \href{https://doi.org/10.1017/CBO9781139060165}{Cambridge University Press}, Cambridge, 2012.

\bibitem{Chu-2017}
Chu C.-H., Infinite dimensional {J}ordan algebras and symmetric cones,
 \href{https://doi.org/10.1016/j.jalgebra.2017.08.017}{\textit{J.~Algebra}} \textbf{491} (2017), 357--371, \href{https://arxiv.org/abs/1707.03610}{arXiv:1707.03610}.

\bibitem{C-DC-I-L-M-2017}
Ciaglia F.M., Cosmo F.D., Ibort A., Laudato M., Marmo G., Dynamical vector
 fields on the manifold of quantum states, \href{https://doi.org/10.1142/S1230161217400030}{\textit{Open Syst. Inf. Dyn.}}
 \textbf{24} (2017), 1740003, 38~pages, \href{https://arxiv.org/abs/1707.00293}{arXiv:1707.00293}.

\bibitem{C-DC-L-M-M-V-V-2018}
Ciaglia F.M., Di~Cosmo F., Laudato M., Marmo G., Mele F.M., Ventriglia F.,
 Vitale P., A pedagogical intrinsic approach to relative entropies as
 potential functions of quantum metrics: the {$q$}-{$z$} family, \href{https://doi.org/10.1016/j.aop.2018.05.015}{\textit{Ann.
 Physics}} \textbf{395} (2018), 238--274, \href{https://arxiv.org/abs/1711.09769}{arXiv:1711.09769}.

\bibitem{C-J-S-2020-02}
Ciaglia F.M., Jost J., Schwachh\"ofer L., Differential geometric aspects of
 parametric estimation theory for states on finite-dimensional
 {$C^*$}-algebras, \href{https://doi.org/10.3390/e22111332}{\textit{Entropy}} \textbf{22} (2020), 1332, 30~pages,
 \href{https://arxiv.org/abs/2010.14394}{arXiv:2010.14394}.

\bibitem{C-J-S-2020}
Ciaglia F.M., Jost J., Schwachh\"ofer L., From the {J}ordan product to
 {R}iemannian geometries on classical and quantum states, \href{https://doi.org/10.3390/e22060637}{\textit{Entropy}}
 \textbf{22} (2020), 637, 27~pages, \href{https://arxiv.org/abs/2005.02023}{arXiv:2005.02023}.

\bibitem{C-DN-J-S-2023}
Ciaglia F.M., Nocera F.D., Jost J., Schwachh\"ofer L., Parametric models and
 information geometry on {$W^{*}$}-algebras, \href{https://doi.org/10.1007/s41884-022-00094-6}{\textit{Inf. Geom.}} \textbf{5}
 (2023), 1--26, \href{https://arxiv.org/abs/2207.09396}{arXiv:2207.09396}.


\bibitem{Dittmann-1993}
Dittmann J., On the {R}iemannian geometry of finite-dimensional mixed states,
 \textit{Sem. Sophus Lie} \textbf{3} (1993), 73--87.

\bibitem{Dittmann-1995}
Dittmann J., On the {R}iemannian metric on the space of density matrices,
 \href{https://doi.org/10.1016/0034-4877(96)83627-5}{\textit{Rep. Math. Phys.}} \textbf{36} (1995), 309--315.

\bibitem{F-F-M-P-2014}
Facchi P., Ferro L., Marmo G., Pascazio S., Defining quantumness via the
 {J}ordan product, \href{https://doi.org/10.1088/1751-8113/47/3/035301}{\textit{J.~Phys.~A}} \textbf{47} (2014), 035301, 9~pages,
 \href{https://arxiv.org/abs/1309.4635}{arXiv:1309.4635}.

\bibitem{F-K-M-M-S-V-2010}
Facchi P., Kulkarni R., Man'ko V.I., Marmo G., Sudarshan E.C.G., Ventriglia F.,
 Classical and quantum {F}isher information in the geometrical formulation of
 quantum mechanics, \href{https://doi.org/10.1016/j.physleta.2010.10.005}{\textit{Phys. Lett.~A}} \textbf{374} (2010), 4801--4803,
 \href{https://arxiv.org/abs/1009.5219}{arXiv:1009.5219}.

\bibitem{F-F-I-M-2013}
Falceto F., Ferro L., Ibort A., Marmo G., Reduction of {L}ie--{J}ordan {B}anach
 algebras and quantum states, \href{https://doi.org/10.1088/1751-8113/46/1/015201}{\textit{J.~Phys.~A}} \textbf{46} (2013), 015201,
 14~pages, \href{https://arxiv.org/abs/1202.3969}{arXiv:1202.3969}.

\bibitem{F-K-1994}
Faraut J., Kor\'anyi A., Analysis on symmetric cones, \textit{Oxford Math. Monogr.}, The
 Clarendon Press, Oxford, 1994.

\bibitem{Fisher-1922}
Fisher R.A., On the mathematical foundations of theoretical statistics,
 \href{https://doi.org/10.1098/rsta.1922.0009}{\textit{Philos. Trans. Roy. Soc.~A}} \textbf{222} (1922), 309--368.

\bibitem{Fuchs-1996}
Fuchs C.A., Distinguishability and accessible information in quantum theory,
 Ph.D.~Thesis, {U}niversite de {M}ontreal, 1996, \href{https://arxiv.org/abs/quant-ph/9601020}{arXiv:quant-ph/9601020}.

\bibitem{Fujiwara-2023}
Fujiwara A., Hommage to {C}hentsov's theorem, \href{https://doi.org/10.1007/s41884-022-00077-7}{\textit{Inf. Geom.}}, {t}o appear.

\bibitem{Hasegawa-1995}
Hasegawa H., Non-commutative extension of the information geometry, in Quantum
 Communications and Measurement ({N}ottingham, 1994), \href{https://doi.org/10.1007/978-1-4899-1391-331}{Plenum}, New York, 1995,
 327--337.

\bibitem{H-P-1997}
Hasegawa H., Petz D., Non-commutative extension of information geometry~{II},
 in Quantum Communication, Computing, and Measurement, \href{https://doi.org/10.1007/978-1-4615-5923-8_12}{Springer}, New York,
 1997, 109--118.

\bibitem{Helstrom-1967}
Helstrom C.W., Minimum mean-squared error of estimates in quantum statistics,
 \href{https://doi.org/10.1016/0375-9601(67)90366-0}{\textit{Phys. Lett.~A}} \textbf{25} (1967), 101--102.

\bibitem{Helstrom-1968}
Helstrom C.W., The minimum variance of estimates in quantum signal detection,
 \href{https://doi.org/10.1109/tit.1968.1054108}{\textit{IEEE Trans. Inform. Theory}} \textbf{14} (1968), 234--242.

\bibitem{Helstrom-1969}
Helstrom C.W., Quantum detection and estimation theory, \href{https://doi.org/10.1007/BF01007479}{\textit{J.~Stat. Phys.}}
 \textbf{1} (1969), 231--252.

\bibitem{Helstrom-1976}
Helstrom C.W., Quantum detection and estimation theory, Academic Press, New
 York, 1976.

\bibitem{H-N-O-1995}
Hilgert J., Neeb K.-H., {\O}rsted B., The geometry of nilpotent coadjoint orbits
 of convex type in {H}ermitian {L}ie algebras, \textit{J.~Lie Theory}
 \textbf{4} (1994), 185--235.

\bibitem{H-N-O-1996}
Hilgert J., Neeb K.-H., {\O}rsted B., Conal {H}eisenberg algebras and associated
 {H}ilbert spaces, \href{https://doi.org/10.1515/crll.1996.474.67}{\textit{J.~Reine Angew. Math.}} \textbf{474} (1996),
 67--112.

\bibitem{H-N-O-1996-2}
Hilgert J., Neeb K.-H., {\O}rsted B., Unitary highest weight representations via
 the orbit method. {I}.~{T}he scalar case, \href{https://doi.org/10.1007/BF00116520}{\textit{Acta Appl. Math.}} \textbf{44} (1996), 151--184.

\bibitem{Iordanescu-2011}
Iord\u{a}nescu R., Jordan structures in geometry and physics, Editura Academiei
 Rom\^ane, Bucharest, 2003, \href{https://arxiv.org/abs/1106.4415}{arXiv:1106.4415}.

\bibitem{Jencova-2003}
Jen\v{c}ov\'a A., Affine connections, duality and divergences for a von
 {N}eumann algebra, \href{https://arxiv.org/abs/math-ph/0311004}{arXiv:math-ph/0311004}.

\bibitem{Jencova-2006}
Jen\v{c}ov\'a A., A construction of a nonparametric quantum information
 manifold, \href{https://doi.org/10.1016/j.jfa.2006.02.007}{\textit{J.~Funct. Anal.}} \textbf{239} (2006), 1--20,
 \href{https://arxiv.org/abs/math-ph/0511065}{arXiv:math-ph/0511065}.

\bibitem{J-vN-W-1934}
Jordan P., von Neumann J., Wigner E.P., On an algebraic generalization of the
 quantum mechanical formalism, \href{https://doi.org/10.1007/978-3-662-02781-321}{\textit{Ann. of Math.}} \textbf{35} (1934),
 29--64.

\bibitem{Kakutani-1948}
Kakutani S., On equivalence of infinite product measures, \href{https://doi.org/10.2307/1969123}{\textit{Ann. of
 Math.}} \textbf{49} (1948), 214--224.

\bibitem{Kirillov-1962}
Kirillov A.A., Unitary representations of nilpotent {L}ie groups,
 \href{https://doi.org/10.1070/RM1962v017n04ABEH004118}{\textit{Russian Math. Surveys}} \textbf{17} (1962), 53--104.

\bibitem{Kirillov-1976}
Kirillov A.A., Elements of the theory of representations, \textit{Grundlehren Math. Wiss.}, Vol.~220, \href{https://doi.org/10.1007/978-3-642-66243-0}{Springer}, Berlin, 1976.

\bibitem{Kirillov-2001}
Kirillov A.A., Geometric quantization, in Dynamical {S}ystems,~{IV},
 \textit{Encyclopaedia Math. Sci.}, Vol.~4, \href{https://doi.org/10.1007/978-3-662-06791-8_2}{Springer}, Berlin, 2001, 139--176,
 \href{https://arxiv.org/abs/1801.02307}{arXiv:1801.02307}.

\bibitem{Kirillov-2004}
Kirillov A.A., Lectures on the orbit method, \textit{Grad. Stud. Math.},
 Vol.~64, \href{https://doi.org/10.1090/gsm/064}{American Mathematical Society}, Providence, RI, 2004.

\bibitem{Koecher-1999}
Koecher M., The {M}innesota notes on {J}ordan algebras and their applications,
 \textit{Lecture Notes in Math.}, Vol.~1710, \href{https://doi.org/10.1007/BFb0096285}{Springer}, Berlin, 1999.

\bibitem{Kostant-1970}
Kostant B., Quantization and unitary representations. {I}.~{P}requantization,
 in Lectures in {M}odern {A}nalysis and {A}pplications,~{III}, \textit{Lecture
 Notes in Math.}, Vol. 170, \href{https://doi.org/10.1007/BFb0079068}{Springer}, Berlin, 1970, 87--208.

\bibitem{Kostecki-2011}
Kostecki R.P., Quantum theory as inductive inference, \href{https://doi.org/10.1063/1.3573636}{\textit{J.~Math. Phys.}}
 \textbf{1305} (2011), 33--40, \href{https://arxiv.org/abs/1009.2423}{arXiv:1009.2423}.

\bibitem{L-L-2022}
Larotonda G., Luna J., On the structure group of an infinite dimensional
 JB-algebra, \href{https://doi.org/10.1016/j.jalgebra.2023.02.003}{\textit{J.~Algebra}} \textbf{622} (2023), 366--403,
 \href{https://arxiv.org/abs/2206.05320}{arXiv:2206.05320}.

\bibitem{L-R-1999}
Lesniewski A., Ruskai M.B., Monotone {R}iemannian metrics and relative entropy
 on noncommutative probability spaces, \href{https://doi.org/10.1063/1.533053}{\textit{J.~Math. Phys.}} \textbf{40}
 (1999), 5702--5724, \href{https://arxiv.org/abs/math-ph/9808016}{arXiv:math-ph/9808016}.

\bibitem{Lichnerowicz-1977}
Lichnerowicz A., Les vari\'et\'es de {P}oisson et leurs alg\`ebres de {L}ie
 associ\'ees, \href{https://doi.org/10.4310/jdg/1214433987}{\textit{J.~Differential Geometry}} \textbf{12} (1977), 253--300.

\bibitem{L-Y-L-W-2020}
Liu J., Yuan H., Lu X.-M., Wang X., Quantum {F}isher information matrix and
 multiparameter estimation, \href{https://doi.org/10.1088/1751-8121/ab5d4d}{\textit{J.~Phys.~A}} \textbf{53} (2020), 023001,
 68~pages, \href{https://arxiv.org/abs/1907.08037}{arXiv:1907.08037}.

\bibitem{Mahalanobis-1936}
Mahalanobis P.C., On the generalized distance in statistics, \textit{Proc.
 Natl. Inst. Sci. India} \textbf{2} (1936), 49--55.

\bibitem{M-M-V-V-2017}
Man'ko V.I., Marmo G., Ventriglia F., Vitale P., Metric on the space of quantum
 states from relative entropy. {T}omographic reconstruction,
 \href{https://doi.org/10.1088/1751-8121/aa7d7d}{\textit{J.~Phys.~A}} \textbf{50} (2017), 335302, 29~pages,
 \href{https://arxiv.org/abs/1612.07986}{arXiv:1612.07986}.

\bibitem{C-M-1991}
Morozowa E.A., Cencov N.N., Markov invariant geometry on state manifolds,
 \href{https://doi.org/10.1007/BF01095975}{\textit{J.~Sov. Math.}} \textbf{56} (1991), 2648--2669.

\bibitem{Niestegge-2020}
Niestegge G., A simple and quantum-mechanically motivated characterization of
 the formally real {J}ordan algebras, \href{https://doi.org/10.1098/rspa.2019.0604}{\textit{Proc.~A.}} \textbf{476} (2020),
 20190604, 14~pages, \href{https://arxiv.org/abs/2019.0604}{arXiv:2019.0604}.

\bibitem{Paris-2009}
Paris M.G.A., Quantum estimation for quantum technology, \href{https://doi.org/10.1142/S0219749909004839}{\textit{Int.~J.
 Quantum Inf.}} \textbf{7} (2009), 125--137, \href{https://arxiv.org/abs/0804.2981}{arXiv:0804.2981}.

\bibitem{Petz-1996}
Petz D., Monotone metrics on matrix spaces, \href{https://doi.org/10.1016/0024-3795(94)00211-8}{\textit{Linear Algebra Appl.}}
 \textbf{244} (1996), 81--96.

\bibitem{Provenzi-2020}
Provenzi E., Geometry of color perception. {P}art 1: structures and metrics of
 a homogeneous color space, \href{https://doi.org/10.1186/s13408-020-00084-x}{\textit{J.~Math. Neurosci.}} \textbf{10} (2020), 7,
 19~pages.

\bibitem{Rao-1945}
Rao C.R., Information and accuracy attainable in the estimation of statistical
 parameters, in Bulletin of the Calcutta Mathematical Society,
 \textit{Springer Ser. Statist.}, Vol.~37, \href{https://doi.org/10.1007/978-1-4612-0919-5_16}{Springer}, Berlin, 1992,
 235--247.

\bibitem{Resnikoff-1974}
Resnikoff H.L., Differential geometry and color perception, \href{https://doi.org/10.1007/BF00275798}{\textit{J.~Math.
 Biol.}} \textbf{1} (1974), 97--131.

\bibitem{Safranek-2017}
\v{S}afr\'anek D., Discontinuities of the quantum Fisher information and the
 Bures metric, \href{https://doi.org/10.1103/PhysRevA.95.052320}{\textit{Phys. Rev.~A}} \textbf{95} (2017), 052320, 13~pages,
 \href{https://arxiv.org/abs/1612.04581}{arXiv:1612.04581}.

\bibitem{Safranek-2018}
\v{S}afr\'anek D., Simple expression for the quantum Fisher information matrix,
 \href{https://doi.org/10.1103/PhysRevA.97.042322}{\textit{Phys. Rev.~A}} \textbf{97} (2018), 042322, 6~pages,
 \href{https://arxiv.org/abs/1801.00945}{arXiv:1801.00945}.

\bibitem{S-A-G-P-2020}
Seveso L., Albarelli F., Genoni M.G., Paris M.G.A., On the discontinuity of the
 quantum {F}isher information for quantum statistical models with parameter
 dependent rank, \href{https://doi.org/10.1088/1751-8121/ab599b}{\textit{J.~Phys.~A}} \textbf{53} (2020), 02LT01, 13~pages,
 \href{https://arxiv.org/abs/1906.06185}{arXiv:1906.06185}.

\bibitem{Souriau-1970}
Souriau J.-M., Structure des syst\`emes dynamiques, Dunod, Paris, 1970.

\bibitem{Sussmann-1973}
Sussmann H.J., Orbits of families of vector fields and integrability of
 distributions, \href{https://doi.org/10.2307/1996660}{\textit{Trans. Amer. Math. Soc.}} \textbf{180} (1973),
 171--188.

\bibitem{S-Y-H-2020}
Suzuki J., Yang Y., Hayashi M., Quantum state estimation with nuisance
 parameters, \href{https://doi.org/10.1088/1751-8121/ab8b78}{\textit{J.~Phys.~A}} \textbf{53} (2020), 453001, 61~pages,
 \href{https://arxiv.org/abs/1911.02790}{arXiv:1911.02790}.

\bibitem{T-A-2014}
T\'oth G., Apellaniz I., Quantum metrology from a quantum information science
 perspective, \href{https://doi.org/10.1088/1751-8113/47/42/424006}{\textit{J.~Phys.~A}} \textbf{47} (2014), 424006, 39~pages,
 \href{https://arxiv.org/abs/1405.487}{arXiv:1405.4878}.

\bibitem{Tulczyjew-1974}
Tulczyjew W.M., Poisson brackets and canonical manifolds, \textit{Bull. Acad.
 Polon. Sci. S\'er. Sci. Math. Astronom. Phys.} \textbf{22} (1974), 931--935.

\bibitem{Uhlmann-1976}
Uhlmann A., The ``transition probability'' in the state space of a
 {$^*$}-algebra, \href{https://doi.org/10.1016/0034-4877(76)90060-4}{\textit{Rep. Math. Phys.}} \textbf{9} (1976), 273--279.

\bibitem{Uhlmann-1992}
Uhlmann A., The metric of bures and the geometric phase, in Groups and Related
 Topics, \textit{Math. Phys. Stud.}, Vol.~13, \href{https://doi.org/10.1007/978-94-011-2801-8_23}{Springer}, Dordrecht, 1992,
 267--274.

\bibitem{Upmeier-1985}
Upmeier H., Symmetric {B}anach manifolds and {J}ordan {$C^\ast$}-algebras,
 \textit{North-Holland Math. Stud.}, Vol. 104, North-Holland Publishing Co.,
 Amsterdam, 1985.

\bibitem{W-vdW-2022}
Westerbaan B., van~de Wetering J., A computer scientist's reconstruction of
 quantum theory, \href{https://doi.org/10.1088/1751-8121/ac8459}{\textit{J.~Phys.~A}} \textbf{55} (2022), 384002, 52~pages,
 \href{https://arxiv.org/abs/2109.10707}{arXiv:2109.10707}.

\bibitem{Wootters-1981}
Wootters W.K., Statistical distance and {H}ilbert space, \href{https://doi.org/10.1103/PhysRevD.23.357}{\textit{Phys. Rev.~D}}
 \textbf{23} (1981), 357--362.

\end{thebibliography}
\end{document}